\documentclass[9pt]{amsart}
\usepackage{amsmath,amsfonts,amsthm}
\usepackage{color}
\usepackage{verbatim}
\usepackage{eepic,epic}
\usepackage{epsfig,subfigure,epstopdf}
\usepackage[]{graphics}
\usepackage[colorlinks,linktocpage,linkcolor=blue]{hyperref}

\newtheorem{theorem}{Theorem}[section]

\newtheorem{lemma}{Lemma}[section]

\newtheorem{assumption}[theorem]{Assumption}
\newtheorem{remark}{Remark}[section]
\newtheorem{corollary}{Corollary}[section]
\numberwithin{equation}{section}

\newcommand{\ga}{\gamma}
\newcommand{\Ga}{\Gamma}

\newcommand{\fy}{\varphi}

\newcommand{\ep}{\epsilon}

\def\al{{\alpha}}
\def\RR{{\mathbb R}}
\def\II{{(D)}}
\def\ur{u^r}

\def\tf{\widetilde f}

\newcount\icount

\def\DD#1#2{\icount=#1
  \ifnum\icount<1
  \,_{ 0}\kern -.1em D^{#2}_{\kern -.1em x}
  \else
  \,_{x}\kern -.2em D^{#2}_1
  \fi
}

\def\DDRI#1#2{\icount=#1
  \ifnum\icount<1
  \,_{-\infty}^{\kern 1em R}\kern -.2em D^{#2}_{\kern -.1em x}
  \else
  \,_{x}^R \kern -.2em D^{#2}_\infty
  \fi
}

\def\DDR#1#2{\icount=#1
  \ifnum\icount<1
 _{0}^{ \kern -.1em R} \kern -.2em D^{#2}_{\kern -.1em x}
  \else
 _{x}^{ \kern -.1em R} \kern -.2em D^{#2}_{\kern -.1em 1}
  \fi
}

\def\DDCI#1#2{\icount=#1
  \ifnum\icount<1
  \,_{-\infty}^{\kern 1em C}  \kern -.2em D^{#2}_{\kern -.1em x}
  \else
  \,_{x}^C \kern -.2em  D^{#2}_\infty
  \fi
}

\def\DDC#1#2{\icount=#1
  \ifnum\icount<1
  \,_{0}^C \kern -.2em  D^{#2}_{\kern -.1em x}
  \else
  \,_{x}^C \kern -.2em D^{#2}_1
  \fi
}

\def\Hd#1{\widetilde H^{#1}(D)}

\def\Hdi#1#2{\icount=#1
  \ifnum\icount<1
  \widetilde H_{L}^{#2}\II
  \else
  \widetilde H_{R}^{#2}\II
  \fi
}

\begin{document}
\title[FEM with reconstruction for fractional BVP]
{A Finite Element Method with Singularity Reconstruction for
Fractional Boundary Value Problems}
\author {Bangti Jin \and Zhi Zhou}
\address {Department of Computer Science, University College London, Gower Street, London WC1E 6BT, UK ({bangti.jin@gmail.com})}
\address {Department of Mathematics, Texas A\&M University, College Station, TX 77843-3368
({{zzhou@math.tamu.edu}})}

\date{started Sep 21, 2013; today is \today}

\begin{abstract}
We consider a two-point boundary value problem involving a Riemann-Liouville fractional derivative of order $\al\in (1,2)$ in the
leading term on the unit interval $(0,1)$. Generally the standard Galerkin finite element method can only give a low-order convergence
even if the source term is very smooth due to the presence of the singularity term $x^{\al-1}$ in the solution representation.
In order to enhance the convergence, we develop a simple singularity reconstruction strategy by splitting the solution into
a singular part and a regular part, where the former captures explicitly the singularity. We derive a new variational formulation
for the regular part, and establish that the Galerkin approximation of the regular part can achieve a better convergence order
in the $L^2(0,1)$, $H^{\al/2}(0,1)$ and $L^\infty(0,1)$-norms than the standard Galerkin approach, with a convergence rate for the recovered
singularity strength identical with the $L^2(0,1)$ error estimate. The reconstruction approach is very flexible in handling explicit
singularity, and it is further extended to the case of a Neumann type boundary condition on the left end point, which
involves a strong singularity $x^{\al-2}$. Extensive numerical results confirm the theoretical study and efficiency of the
proposed approach.
\end{abstract}

\maketitle

\section{Introduction}\label{sec:intro}
We consider the following fractional-order boundary value problem:
find $u$ such that
\begin{equation}\label{strongp}
  \begin{aligned}
   -\DDR 0 \alpha u + qu = f, & \quad \mbox{ in } D\equiv(0,1),\\
   u(0)=u(1)=0, &
   \end{aligned}
\end{equation}
where $\alpha\in(1,2)$ is the order of the derivative, and $\DDR 0\al$ refers to
the Riemann-Liouville fractional derivative of order $\alpha$ defined in \eqref{Riemann}
below. Here $f$ is a function in $L^2\II$ or other suitable Sobolev space. The
potential coefficient $q \in L^\infty(D)$ is a bounded measurable function.
If the fractional order $\alpha$ equals two, the fractional derivative
$\DDR0\alpha u$ coincides with the usual second-order derivative $u''$ \cite{KilbasSrivastavaTrujillo:2006},
and thus problem \eqref{strongp} generalizes the classical two-point boundary value problem.

The interest in the model \eqref{strongp} is largely motivated by anomalous diffusion processes, in which the
mean square variance grows faster (superdiffusion) or slower (subdiffusion) than that in a
Gaussian process. In recent years, a lot of works \cite{BensonWheatcraftMeerschaert:2000,
MetzlerKlafter:2000} have shown that anomalous diffusion offers a superior
fit to experimental data observed in transport phenomena in some practical applications,
e.g., viscoelastic materials, soil contamination, and underground water flow. The
model \eqref{strongp} represents the steady state of one-dimensional asymmetric superdiffusion
process, which is typically observed in magnetized plasma and geophysical flow
\cite{delCastillo:2003,delCastillo:2005}. It can be viewed as the macroscopic
counterpart of L\'{e}vy flights, like Brownian motion for the classical diffusion equation;
see \cite{BensonWheatcraftMeerschaert:2000} for a detailed derivation from the microscopic model and
relevant physical explanations in the context of underground flow. Numerous experimental
studies have demonstrated that it can capture accurately the distinct features of superdiffusion processes.

The excellent modeling capabilities of the model \eqref{strongp} have generated considerable interest in deriving,
analyzing and testing numerical methods for solving related initial boundary
value problems. The finite difference scheme is predominant in existing studies, and it is usually based on a shifted Gr\"{u}nwald-Letnikov formula, which
is first order accurate, for the Riemann-Liouville fractional derivative in space (see e.g. \cite{TadjeranMeerschaert:2007,
Sousa:2009}). In \cite{DengHesthaven:2013}, a local discontinuous Galerkin method was developed. The shifted
Gr\"{u}nwald-Letnikov formulas can be combined with suitable weights to achieve a second-order accuracy, under the assumption
that the solution is sufficiently smooth \cite{TianZhouDeng:2014}. However, the precise condition under which the solution
is indeed smooth remains unknown. Very recently, we \cite{JinLazarovPasciakZhou:2014siam}
developed a finite element method in space/backward Euler the (or Crank-Nicolson method) in time, based on the variational
formulation developed in \cite{JinLazarovPasciak:2013a}, and provided $L^2\II$ error estimates for the approximation.

The mathematical study on the model \eqref{strongp} has just started to attract attention. First we
note that the Riemann-Liouville fractional derivative operator is not the fractional power of the
Laplace operator, and thus the well developed analytic theory on the fractional Laplacian does not
apply to the model \eqref{strongp}. Ervin and Roop \cite{ErvinRoop:2006} presented a first rigorous
analysis of the stability of the weak formulation of problem \eqref{strongp}, and an optimal
$\Hd{\alpha/2}$-error estimate for the Galerkin finite element method (FEM) was also developed under
the assumption that the solution has full regularity, i.e., $\|u\|_{H^\al\II} \le C \|f \|_{L^2\II}$,
which however is generally not true \cite{JinLazarovPasciak:2013a} (see Section \ref{ssec:fraccal} below for
the definition of the space $\Hd s$). Recently, Wang and Yang \cite{WangYang:2013} developed
a Petrov-Galerkin formulation of the model \eqref{strongp} with a variable coefficient, and analyzed
its variational stability, but the discrete counterpart was not discussed.
The model \eqref{strongp} was very recently revisited in \cite{JinLazarovPasciak:2013a}, and the case of a
Caputo fractional derivative
was also analyzed. Further, the proper variational formulations for both fractional derivatives
were derived, and convergence rates of the Galerkin FEM approximations were established.

In \cite{JinLazarovPasciak:2013a}, it was shown that the solution $u$ to \eqref{strongp} is generally only
in the space $H^{\al-1+\beta}\II$ with $\beta\in [1-\alpha/2,1/2)$ even if the source term $f$ is much smoother
than $L^2\II$, due to the presence of the singular term $x^{\al-1}$ in the solution representation; see
\eqref{eqn:sols} below  for details. This leads to a low-order convergence rate for the standard Galerkin
FEM approximation based on the continuous piecewise linear finite element space (on a quasi-uniform partition
of the domain $D$). There are several possible strategies for improving the convergence, e.g., adaptive
refinement, solution space enrichment and singularity reconstruction. In this paper we opt for a simple
strategy, inspired by the interesting work \cite{CaiKim:2001}, that allows us to overcome this difficulty.
Specifically, we split the solution $u$ into
\begin{equation*}
  u=\ur + \mu u^s,
\end{equation*}
where $\ur$ denotes the regular part of the solution $u$ that has a better Sobolev regularity
than $H^{\al-1+\beta}\II$ with $\beta\in[1-\alpha/2,1/2)$, and $u^s$ captures the leading singularity
$x^{\al-1}$, with the scalar $\mu$ being the singularity strength. We derive a new
variational formulation for the regular part $\ur$ and a reconstruction formula for the scalar
$\mu$, and further, develop a stable finite element scheme for approximating the regular part $\ur$
and then reconstructing the singularity strength $\mu$. Theoretically, in Theorems \ref{lem:femerror}
and \ref{thm:femerrorinf}, for the Galerkin approximation $u_h$, we derive
error estimates in the $\Hd {\al/2}$, $L^2\II$ and $L^\infty\II$-norm, and the recovered singularity
strength $\mu_h$ exhibits a convergence order identical with the $L^2\II$ error estimate. Namely, for $f\in \Hd
\gamma$, $q\in \Hd\gamma\cap L^\infty\II$, with $\beta\in[1-\alpha/2,1/2)$ and
$\ell(\beta,\gamma)=:\min(\al-1+\beta,\ga)$, there holds
\begin{equation*}
   \|u-u_h\|_{L^2\II} + h^{\al/2-1+\beta}\|u-u_h\|_{\Hd{\alpha/2}}
   \le C h^{\min(2,\al+\ell(\beta,\gamma))-1+\beta} \|f\|_{\Hd\ga},
\end{equation*}
which is higher than that for the classical approach, cf. Theorem \ref{thm:femrl} below.
It is worth noting that all error estimates are expressed in terms of the right-hand side $f$ only.
Numerically, the computational effort of the new approach is nearly identical with the classical one.
Further, the singularity reconstruction approach is capable of resolving very strong singularity, which
is highly nontrivial for other approaches, if not impossible at all. A strong solution singularity arises, for
example, in the case of a Neumann boundary condition at the left end point, on which we shall further illustrate the
flexibility of the proposed reconstruction approach.

The rest of the paper is organized as follows. In Section \ref{sec:prelim} we recall some important properties of
fractional derivatives and integrals, and describe the variational formulation for the source problem \eqref{strongp}
and the standard Galerkin FEM. We provide also a new $L^\infty\II$-error estimate, besides the known $\Hd {\alpha/2}$
and $L^2\II$ error estimates. In Section \ref{sec:formulation}, a novel variational formulation for the regular part
$u^r$ of the solution and the reconstruction formula of the singularity strength $\mu$ are developed, and the stability
of the formulation is established. The stability of the discrete variational formulation based on a Galerkin procedure is shown in
Section \ref{sec:newFEM}, and the $\Hd {\alpha/2}$, $L^2\II$ and $L^\infty\II$ error estimates  are also provided. Further,
in Section \ref{sec:neumann}, we extend the approach to the case of a Neumann boundary condition at the left end point,
which involves a strong solution singularity $x^{\al-2}$, to illustrate its flexibility. Finally, numerical
results for are presented in Section \ref{sec:numeric} to confirm the convergence theory. Throughout, we use the notation $C$, with
or without a subscript, to denote a generic constant, which can take different values at different occurrences, but it
is always independent of the solution $u$ and the mesh size $h$.

\section{Preliminaries}\label{sec:prelim}
In this part, we describe fundamentals of fractional calculus, the variational
formulation for the source problem \eqref{strongp} and the Galerkin approximation.
\subsection{Fractional calculus}\label{ssec:fraccal}
We first briefly recall the Riemann-Liouville fractional derivative. For any $\beta>0$
with $n-1 < \beta < n$, $n\in \mathbb{N}$, the left-sided Riemann-Liouville fractional
derivative $\DDR0\beta u$ of order $\beta$ of a function $u\in H^n(D)$ is defined by
\cite[pp. 70]{KilbasSrivastavaTrujillo:2006}:
\begin{equation}\label{Riemann}
  \DDR0\beta u =\frac {d^n} {d x^n} \bigg({_0\hspace{-0.3mm}I^{n-\beta}_x} u\bigg) .
\end{equation}
Here $_0\hspace{-0.3mm}I^{\gamma}_x$ for $\gamma>0$
is the left-sided Riemann-Liouville fractional integral operator of order $\gamma$ defined by
\begin{equation*}
 ({\,_0\hspace{-0.3mm}I^\gamma_x} f) (x)= \frac 1{\Gamma(\gamma)} \int_0^x (x-t)^{\gamma-1} f(t)dt,
\end{equation*}
where $\Gamma(\cdot)$ is Euler's Gamma function defined by $\Gamma(x)=\int_0^\infty t^{x-1}e^{-t}dt$.
The right-sided versions of fractional-order integral $_xI_1^\gamma$ and derivative $\DDR 1 \beta$ are defined analogously by
\begin{equation*}
  ({_x\hspace{-0.3mm}I^\gamma_1} f) (x)= \frac 1{\Gamma(\gamma)}\int_x^1 (t-x)^{\gamma-1}f(t)\,dt\quad\mbox{and}\quad
  \DDR1\beta u =(-1)^n\frac {d^n} {d x^n} \bigg({_x\hspace{-0.3mm}I^{n-\beta}_1} u\bigg) .
\end{equation*}

Now we introduce some function spaces. For any $\beta\ge 0$, we denote $H^\beta\II$ to be
the Sobolev space of order $\beta$ on the unit interval $D$, and $\Hd \beta $ to be the set of
functions in $H^\beta\II$ whose extension by zero to the real line $\mathbb{R}$ are in $H^\beta(\RR)$. Analogously, we define
$\Hdi 0 \beta$ (respectively, $\Hdi 1 \beta$) to be the set of functions $u$ whose extension by zero, denoted by
$\tilde{u}$, is in $H^\beta(-\infty,1)$ (respectively, $H^\beta(0,\infty)$). For $u\in \Hdi 0
\beta$, we set $\|u\|_{\Hdi 0\beta}:=\|\tilde{u}\|_{H^\beta(-\infty,1)}$, and analogously the norm
in $\Hdi 1 \beta$.

The following theorem collects their important properties. The proof of part (a) can be found
in \cite[pp. 46, Theorem 2.4]{SamkoKilbasMarichev:1993} or \cite[pp. 73, Lemma 2.3]{KilbasSrivastavaTrujillo:2006};
and parts (b) and (c) can be found in \cite[Theorems 2.1 and 3.1]{JinLazarovPasciak:2013a}.
\begin{theorem}\label{thm:fracop}
The following statements hold.
\begin{itemize}
  \item[$\mathrm{(a)}$] The integral operators $_0I_x^\beta$ and $_xI_1^\beta$ satisfy the semigroup property, i.e., for any $\beta,\gamma>0$
   \begin{equation*}
     _0I_x^\beta ({_0I_x^\gamma} f) = {_0I_x^{\beta+\gamma}}f\quad \mbox{and}\quad
     _xI_1^\beta ({_xI_1^\gamma} f) = {_xI_1^{\beta+\gamma}}f\quad \forall f\in L^2\II.
   \end{equation*}
  \item[$\mathrm{(b)}$] The operators $\DDR0\beta: \Hdi 0 \beta \to L^2\II$ and $\DDR1\beta: \Hdi 1\beta \to L^2\II $ are continuous.
  \item[$\mathrm{(c)}$] For any $s,\beta\geq 0$, the operator
$_0I_x^\beta$ is bounded from $\Hdi0 s$ to $\Hdi0{\beta+s}$,
and $_xI_1^\beta$ is bounded from $\Hdi1 s$ to $\Hdi1{\beta+s}$.
\end{itemize}
\end{theorem}

We shall also need an ``algebraic'' property of the space $\Hd s$, $0<s<1$
\cite{JinLazarovPasciak:2013a}.
\begin{lemma}\label{lem:Hmulti}
Let $0<s\leq1$, $s\neq 1/2$. Then for any $u\in H^ s\II \cap L^\infty\II$
and $v\in \Hd s\cap L^\infty\II$, $uv\in \Hd s$.
\end{lemma}

\subsection{Variational formulation}
Now we describe the variational formulation of problem \eqref{strongp}, developed in \cite{JinLazarovPasciak:2013a}.
We first consider the simplest source problem with a vanishing potential, i.e., $q \equiv 0$:
\begin{equation}\label{eqn:simp}
\begin{split}
    -{\DDR0 \al} u=f &\quad\mbox { in }D,\\
     u(0)=u(1)=0&,
\end{split}
\end{equation}
with $f\in L^2\II$. Note that ${_0I_x^\al}f \in \Hdi0 \al$ and thus by Theorem
\ref{thm:fracop}(b), $\DDR0 \al \left({_0I_x^\al}f \right) \in L^2\II$ is well-defined.
By Theorem \ref{thm:fracop}(a) we have for all $f\in C_0^{\infty}\II$
\begin{equation*}
 \DDR0 \al \left( {_0I_x^\al}f\right)= \left( {_0I_x^{2-\al}}\left( {_0I_x^\al}f\right)\right)''
   =\left( I^2f\right)''=f,
\end{equation*}
and it is true for $f\in L^2\II$ by a density argument; see also \cite[pp. 44, Theorem 2.4]{SamkoKilbasMarichev:1993}
discussions in the general case. Thus the solution $u$ of problem \eqref{eqn:simp} has the explicit form:
\begin{equation}\label{eqn:sols}
u(x)=-({_0I_x^\al}f)(x)+({_0I_x^\al}f)(1)x^{\al-1},
\end{equation}
by noting the relation $\DDR 0 \alpha x^{\alpha-1}=0$.
The solution $u$ is only in $\Hdi0 {\al-1+\beta}$ with $\beta \in [1-\alpha/2,1/2)$
in general, even for a smooth source term $f$, due to the presence of the singular term $x^{\al-1}$.

The variational formulation of problem \eqref{eqn:simp} is given by \cite{ErvinRoop:2006,JinLazarovPasciak:2013a}:
find $u\in V\equiv \Hd {\al/2}$ such that
\begin{equation}\label{eqn:var1}
    A(u,v)=(f,v) \quad \forall v\in V,
\end{equation}
with the bilinear form $A(\cdot,\cdot)$ defined on $V \times V$ by
\begin{equation}\label{eqn:Abil}
    A(u,v)=-(\DDR0{{\alpha/2}}u,\,\DDR1{{\alpha/2}}v) \quad \forall u,v \in V.
\end{equation}
It is known that the bilinear form $A(\cdot,\cdot)$ is coercive on the space $V$
(see \cite[Lemma 3.1]{ErvinRoop:2006}, \cite[Lemma 4.2]{JinLazarovPasciak:2013a}),
i.e., there is a constant $C_0$ such that for all $u\in V$
\begin{equation}\label{A-coercive}
  A(u,u) \ge C_0 \| u \|^2_{V},
\end{equation}
and it is continuous on $V$, i.e., for all $u,v\in V$
\begin{equation}\label{continuous}
    |A(u,v)| \le C_1 \| u \|_V\| v \|_V.
\end{equation}

We now turn to the general case of $q\neq 0$ and define
\begin{equation*}
a(u,v) = A(u,v) +(qu,v).
\end{equation*}
Then the variational formulation for problem \eqref{strongp} is given
by: to find $u\in V$ such that
\begin{equation}\label{varrl}
    a(u,v)=(f,v)\quad \forall v\in V.
\end{equation}
To study the bilinear form $a(\cdot,\cdot)$, we make the following assumption. The assumption holds
automatically for the case $q\geq 0$, in view of the coercivity of the bilinear form $A(\cdot,\cdot)$
on the space $\Hd {\alpha/2} $ \cite[Lemma 4.2]{JinLazarovPasciak:2013a}.
\begin{assumption} \label{ass:riem}
Let the bilinear form $a(u,v)$ with $u,v\in V$ satisfy
\begin{itemize}
 \item[{$\mathrm{(a)}$}]  The problem of finding $u \in V$ such that $a(u,v)=0$ for all $v \in V$
           has only the trivial solution $u\equiv 0$.
 \item[{$(\mathrm{a}^\ast)$}] The problem of finding $v \in V$ such that $a(u,v)=0$ for all $u \in V$
    has only the trivial solution $v\equiv 0$.
\end{itemize}
\end{assumption}

Under Assumption \ref{ass:riem}, there exists a unique solution $u\in V$
to \eqref{varrl} \cite[Theorem 4.3]{JinLazarovPasciak:2013a}. In fact the variational
solution is a strong solution. To see this, we consider the problem
\begin{equation}\label{rl-qneq0}
-{\DDR0{\alpha}} u = f-qu.
\end{equation}
A strong solution is given by \eqref{eqn:sols} with a right hand side $\widetilde{f}=f-qu$.
It satisfies the variational problem \eqref{varrl} and hence coincides with the unique variational
solution. Further, the solution $u$ satisfies the following regularity
\begin{equation*}
  \|u \|_{\Hdi 0 {\al -1 +\beta}} \le C \|f\|_{L^2\II},
\end{equation*}
for any $\beta\in[1-\alpha/2,1/2)$.
Thus the global regularity of the solution $u$ does not improve with the
regularity of the source term $f$, due to the inherent presence of the term $x^{\alpha-1}$
in the solution representation \eqref{eqn:sols}.

\subsection{Standard Galerkin FEM}\label{sec:FEM}
Now we describe the standard finite element approximation based on a uniform partition of the interval
$D$. Let $h=1/m $ be the mesh size with $m>1$ being a positive integer, and the nodes $x_j=jh$, $j=0,
\ldots,m$. We then define $V_h$ to be the continuous piecewise linear finite element space, i.e.,
\begin{equation*}
   V_h = \left\{\chi\in C_0(\overline{D}): \chi \mbox{ is linear over }  [x_j,x_{j+1}], \, j=0,\ldots,m-1 \right\}.
\end{equation*}
It is well known that the space $V_h$ has the following approximation properties. 
\begin{lemma}\label{fem-interp-U}
If $u \in H^\gamma\II \cap \Hd {\al/2}$ with $ \alpha/2 \le \gamma \le 2$, then
\begin{equation}\label{approx-Uh}
\inf_{v \in V_h} \| u -v \|_{H^{\al/2}\II} \le ch^{\gamma -\alpha/2} \|u\|_{H^\gamma\II}.
\end{equation}
\end{lemma}

The Galerkin FEM problem is to find $u_h\in V_h$ such that
\begin{equation}\label{eqn:FEM1}
    a(u_h,v)=(f,v) \quad \forall v \in V_h.
\end{equation}
The discrete approximation $u_h$ exists and is unique for small $h$, and further, it
satisfies the following error estimates \cite[Theorem
5.2]{JinLazarovPasciak:2013a}. We note that the theorem has the flavor of Mikhlin's
theorem, in the sense of ``asymptotic stability''.
\begin{theorem}\label{thm:femrl}
Let Assumption \ref{ass:riem} hold, $f\in L^2\II$, and $q\in L^\infty\II$.
Then there is an $h_0>0$ such that for all $h\le h_0$, problem \eqref{eqn:FEM1}
has a unique solution $u_h\in V_h$ and it satisfies for any $\beta\in[1-\alpha/2,1/2)$
\begin{equation*}
   \|u-u_h\|_{L^2\II} + h^{\alpha/2-1+\beta}\|u-u_h\|_{H^{\alpha/2}\II}
   \le Ch^{\alpha-2+2\beta} \|f\|_{L^2\II}.
\end{equation*}
\end{theorem}
\begin{remark}\label{rmk:conv-standard}
Due to the inherent presence of the singular term $x^{\alpha-1}$, the solution $u$ has rather low global
regularity, especially for $\alpha$ close to unity, even if the source term $f$ is smooth. Hence, the convergence rate of the Galerkin approximation
$u_h$ based on the formulation \eqref{eqn:FEM1} is slow: for $\alpha$ close to unity, the theoretical
rate is close to zero. This necessitates the development of new techniques with enhanced convergence rates.
\end{remark}

Next we derive a new $L^\infty\II$-error estimate for the Galerkin approximation $u_h$. To this end, we first
recall Green's function to the adjoint problem of \eqref{strongp}, i.e., for all $x\in D$
\begin{equation}
\begin{split}
  - {_{y}^{ \kern -.1em R} \kern -.2em D^{\alpha}_{\kern -.1em 1}} G(x,y) + q(y) G(x,y) & = \delta_x(y),\ \ y \in D,\\
  G(x,0)=G(x,1)&=0.
\end{split}
\end{equation}
For the case $q\equiv0$, $G(x,y)$ is available in closed form \cite{BaiLu:2005}.
It follows from the solution representation \eqref{eqn:sols} that
$G(x,y)$ is given explicitly by
\begin{equation}\label{eqn:greenq0}
    G(x,y) = \left\{\begin{array}{ll}\displaystyle
       \frac{(1-y)^{\al-1}x^{\al-1}}{\Ga(\al)}, & 0\le x\le y \le 1;\\[1.6ex]
       \displaystyle
       \frac{(1-y)^{\al-1}x^{\al-1}-(x-y)^{\al-1}}{\Ga(\al)}, & 0\le y\le x \le 1.\\
       \end{array}\right.
\end{equation}
We note that $G(x,\cdot)\in \Hd{\al-1+\beta}$ with $\beta\in[1-\alpha/2,1/2)$.
In fact, the fractional integral ${_yI_1^\alpha}\delta_x(y)$ of the Dirac-delta function
$\delta_x(y)$ is given by
\begin{equation}\label{eqn:intdelta}
    ({_yI_1^{\al}}\delta_x)(y) = \left\{\begin{array}{ll}
      \displaystyle  0, & 0\le x\le y \le 1;\\
      \displaystyle \frac{(x-y)^{\al-1}}{\Ga(\al)}, & 0\le y\le x \le 1.\\
       \end{array}\right.
\end{equation}
which lies in $\widetilde H^{\al-1+\beta}\II$
 with $\beta\in[1-\alpha/2,1/2)$. In the case of a general potential $q\neq 0$, the
weak formulation is to find $G(x,\cdot)\in V$ such that
\begin{equation}\label{eqn:weakgreen}
   a(v, G(x,\cdot))= \langle v,\delta_x\rangle=v(x) \quad \forall v \in V.
\end{equation}
By Sobolev embedding theorem \cite{AdamsFournier:2003}, $\delta_x\in V^*$, the dual space of $V$, and thus
the existence and uniqueness of $G(x,y)\in \Hd{\al/2}$
follows directly from the stability of the variational formulation. Moreover, it satisfies the differential equation
\begin{equation*}
   -{_{y}^{ \kern -.1em R} \kern -.2em D^{\alpha}_{\kern -.1em 1}} G(x,y)   = \delta_x(y)- q(y) G(x,y),\quad y \in  D.
\end{equation*}
Then the facts that $q(\cdot)G(x,\cdot)\in L^2\II$ and ${_yI_1^{\al}}\delta_x \in \Hdi1{\al-1+\beta}$
lead to the desired regularity $G(x,\cdot) \in \Hdi1{\al-1+\beta}$ with $\beta\in[1-\alpha/2,1/2)$.

Now we can state an $L^\infty\II$-error estimate of the Galerkin approximation $u_h$.
\begin{theorem}\label{thm:linf}
Let $q\in L^\infty\II$, Assumption \ref{ass:riem} hold, and $f\in L^2\II$.
Then there is an $h_0>0$ such that for all $h\le h_0$,
the solution $u_h$ to problem \eqref{eqn:FEM1} satisfies for any $\beta\in[1-\alpha/2,1/2)$
\begin{equation*}
   \|u-u_h\|_{L^\infty\II}  \le Ch^{\alpha-2+2\beta} \|f\|_{L^2\II}.
\end{equation*}
\end{theorem}
\begin{proof}
Using the weak formulation \eqref{eqn:weakgreen} of $G(x,y)$
and the Galerkin orthogonality, we have for any $w_h\in V_h$
\begin{equation*}
   (u-u_h)(x)= a(u-u_h, G(x,\cdot))=a(u-u_h, G(x,\cdot)-w_h).
\end{equation*}
Then applying \eqref{continuous},
Lemma \ref{fem-interp-U} and Theorem \ref{thm:femrl}, we obtain for any $\beta\in[1-\alpha/2,1/2)$
\begin{equation*}
\begin{split}
|(u-u_h)(x)|&\le C\| u-u_h\|_V\inf_{w_h\in V_h} \|G(x,\cdot)-w_h\|_V \\
  &\le Ch^{\alpha-2+2\beta}\| u \|_{\Hdi0{\alpha-1+\beta}} \le Ch^{\alpha-2+2\beta}\| f \|_{L^2\II}.
\end{split}
\end{equation*}
This completes the proof of the theorem.
\end{proof}

\begin{remark}
It is well known that for the standard second-order two-point boundary value problem, the Galerkin
approximation $u_h$ exhibits superconvergence at the nodes due to the piecewise smoothness of
Green's function \cite{DouglasDupont:1974}. For the fractional counterpart, the Green's function is no
longer piecewise smooth: it is only in $ \widetilde H^{\al-1+\beta}_L\II$ with
$\beta\in[1-\alpha/2,1/2)$ even for $q=0$. Our numerical experiments in Section
\ref{sec:numeric} indicate that the $L^\infty\II$-error estimate in Theorem \ref{thm:linf} is sharp.
\end{remark}

\section{A new variational formulation and well-posedness}\label{sec:formulation}
Now we develop a new approach for the source problem \eqref{strongp} based on singularity splitting,
which is inspired by the interesting work \cite{CaiKim:2001}, developed for the Poisson equation on
the L-shaped domain. We shall derive a new variational formulation for the regular part, and establish
its stability and the enhanced regularity of the regular part.

\subsection{Derivation of the new approach}
The new approach is based on splitting the solution $u$ to problem \eqref{strongp} into a regular
part $\ur$ and a singular part involving $x^{\al-1}$:
\begin{equation*}
  u(x)=\ur + \mu \left(x^{\al-1}-x^2\right).
\end{equation*}
In the splitting, the choice of $x^2$ is arbitrary. We shall assume
\begin{equation*}
  {_0\hspace{-0.3mm}I^\al_x} (q(x^{\al-1}-x^2)) (1) \neq -1.
\end{equation*}
If this condition does not hold, we may replace the choice $x^2$ by any other function $v$ in the space $\Hdi0s$, $s\geq2$, with $v(1)=1$, such
that ${_0\hspace{-0.3mm}I^\al_x} (q(x^{\al-1}-v)) (1) \neq -1$. Now we develop a new variational
formulation to uniquely determine the regular part $\ur$. To motivate this, we deduce from
\eqref{eqn:sols} and \eqref{rl-qneq0} that the solution $u$ can be split into
\begin{equation*}
\begin{split}
u(x) = &-{_0\hspace{-0.3mm}I^\al_x} (f-qu) (x)+ {_0\hspace{-0.3mm}I^\al_x} (f-qu) (1)x^{\al-1}\\
     = &-{_0\hspace{-0.3mm}I^\al_x} (f-qu) (x) + {_0\hspace{-0.3mm}I^\al_x} (f-qu)(1) x^2\\
     &+{_0\hspace{-0.3mm}I^\al_x} (f-qu)(1)(x^{\al-1}- x^2).
\end{split}
\end{equation*}
Let $\mu={_0\hspace{-0.3mm}I^\al_x} (f-qu)(1)$.
We can set the regular part $u^r$ and singular part $u^s$ to
\begin{equation}\label{eqn:reg+sing}
u^r(x)=-{_0\hspace{-0.3mm}I^\al_x} (f-qu) (x) + \mu x^2
\quad\text{and}\quad u^s(x)=x^{\al-1}- x^2,
\end{equation}
respectively. Then obviously,
\begin{equation*}
  u(x)=u^r(x)+\mu u^s(x).
\end{equation*}

To construct a variational problem for the regular part $\ur $, we first derive a reconstruction
formula for the singularity strength $\mu$ in terms of $\ur$. By the definition of $\mu$ and the regular part
$\ur$, cf. \eqref{eqn:reg+sing}, we have
\begin{equation*}
   \mu={_0\hspace{-0.3mm}I^\al_x} \left(f-q(u^r+\mu (x^{\al-1}- x^2))\right) (1).
\end{equation*}
By rearranging terms, $\mu$ can be written as
\begin{equation}\label{eqn:lambda}
  \mu=c_0 \left({_0\hspace{-0.3mm}I^\al_x} (f-qu^r)\right) (1),
\end{equation}
where the constant $c_0$ is defined by
\begin{equation}\label{eqn:c0}
  c_0=\frac{1}{1+{_0\hspace{-0.3mm}I^\al_x} (q(x^{\al-1}-x^2)) (1)}.
\end{equation}
Hence the solution $u$ of problem \eqref{strongp} can be split into
\begin{equation*}
   u=\ur+c_0 \left({_0\hspace{-0.3mm}I^\al_x} (f-q\ur)\right) (1)(x^{\al-1}-x^2).
\end{equation*}
Upon substituting it back into \eqref{strongp} and setting
\begin{equation*}
   c_1(x)= {\DDR 0 \al \left(x^{\alpha-1}-x^2\right)}= -\frac{2}{\Gamma(3-\al)}x^{2-\al},
\end{equation*}
we arrive at the following fractional integro-differential problem for the regular part $\ur$
\begin{equation}\label{eqn:eqnforreg}
  \begin{aligned}
    -\DDR 0 \alpha \ur + q\ur  + Q \left({_0\hspace{-0.3mm}I^\al_x}(q\ur)\right)& (1)  = \tf   \quad \mbox{ in }D,\\
   \ur(0)=\ur(1)=0, &
   \end{aligned}
\end{equation}
where the functions $Q(x)$ and $\tf(x)$ are defined respectively by
\begin{equation}
  \begin{aligned}\label{eqn:Q+tf}
     Q(x)&=c_0c_1(x)-c_0q(x)(x^{\al-1}-x^2)\in L^{\infty}\II,\\
    \tf(x)&=f(x)+c_0c_1(x)\left({_0\hspace{-0.3mm}I^\al_x}f\right)(1)-
       c_0\left({_0\hspace{-0.3mm}I^\al_x}f\right)(1)q(x)(x^{\al-1}-x^2)\in L^2\II.
  \end{aligned}
\end{equation}

\subsection{A new variational formulation and its stability}
For problem \eqref{eqn:eqnforreg} for the regular part $u^r$, we introduce the following bilinear form
\begin{equation}\label{eqn:ar}
    a_r(u,v)= A(u,v)+b(u,v)\quad \forall u,v\in V,
\end{equation}
with the form $b(\cdot,\cdot)$ given by
\begin{equation}\label{eqn:b}
    b(u,v)= (qu,v)+ {_0\hspace{-0.3mm}I^\al_x}(qu)(1)(Q,v)\quad \forall u,v\in V.
\end{equation}
Using Theorem \ref{thm:fracop}(c), $q\in L^{\infty}\II$, and Sobolev embedding theorem
\cite{AdamsFournier:2003}, we deduce
\begin{equation}\label{eqn:bconti}
\begin{split}
   |b(u,v)|  &\le |{_0\hspace{-0.3mm}I^\al_x}(qu)(1)||(Q,v)| + \|qu\|_{L^2\II}\|v\|_{L^2\II}\\
         &\le C \| {_0\hspace{-0.3mm}I^\al_x}(qu)\|_{\Hdi0{\al}}\|Q\|_{L^2\II}\|v\|_{L^2\II} + C\|u\|_{L^2\II}\|v\|_{L^2\II}\\
         & \le C \|u\|_{L^2\II} \|v\|_{L^2\II}.
\end{split}
\end{equation}
That is, $b(\cdot,\cdot)$ is continuous on $L^2\II \times L^2\II$. Thus
the bilinear form $a_r(\cdot,\cdot)$ is continuous on $V\times V$,
i.e., there exists a constant $C > 0$ such that
\begin{equation}\label{eqn:arcontinuous}
    a_r(u,v)\le C \|u \|_V \| v\|_V \quad \forall u,v\in V.
\end{equation}

The variational problem for the regular part $\ur$ reads: find $\ur \in V$ satisfying
\begin{equation}\label{eqn:var2}
 a_r(\ur,v)= (\tf,v) \quad \forall v\in V.
\end{equation}
Once the regular part $u^r$ is determined, the singularity strength $\mu$ can be reconstructed
using formula \eqref{eqn:lambda}. Finally the solution $u$ of the source problem \eqref{strongp}
has the following representation:
\begin{equation}\label{eqn:ucompose}
  u=\ur+\mu(x^{\al-1}-x^2).
\end{equation}

Now we turn to the well-posedness of the variational formulation \eqref{eqn:var2}.
In case of $q=0$, the bilinear form $a_r(\cdot,\cdot)$ is identical with $a(\cdot,\cdot)$.
Thus the well-posednees of problem \eqref{eqn:var2} follows directly from the continuity
and coercivity of the bilinear form. It suffices to consider the case $q\neq 0$. To
this end, we make the following uniqueness assumption on the bilinear form $a_r(\cdot,\cdot)$.
\begin{assumption} \label{ass:riem2}
Let the bilinear form $a_r(u,v)$ with $u,v\in V$ satisfy
\begin{itemize}
 \item[{$\mathrm{(a)}$}]  The problem of finding $u \in V$ such that $a_r(u,v)=0$ for all $v \in V$
           has only the trivial solution $u\equiv 0$.
 \item[{$(\mathrm{a}^\ast)$}] The problem of finding $v \in V$ such that $a_r(u,v)=0$ for all $u \in V$
    has only the trivial solution $v\equiv 0$.
\end{itemize}
\end{assumption}

The next result shows that Assumption \ref{ass:riem} implies Assumption \ref{ass:riem2}(a).
However, the connection between Assumptions \ref{ass:riem} and \ref{ass:riem2}(a$^\ast$) is still unclear.
\begin{lemma}
Assumption \ref{ass:riem} implies Assumption \ref{ass:riem2} $\mathrm{(a)}$.
\end{lemma}
\begin{proof}
If $\tf\equiv0$, then by the definition of $\tf$ in \eqref{eqn:Q+tf} we have
\begin{equation*}
  f(x)=-c_0c_1(x)\left({_0\hspace{-0.3mm}I^\al_x}f\right)(1)+
       c_0\left(\left({_0\hspace{-0.3mm}I^\al_x}f\right)(1)q(x)(x^{\al-1}-x^2)\right).
\end{equation*}
We observe that $f$ has an explicit representation for any $c\in \mathbb{R}$
\begin{equation*}
  f=c\left(-c_1(x)+ q(x)(x^{\al-1}-x^2)\right).
\end{equation*}
Now suppose that there exists a $\chi\in V$, $\chi\neq 0$, such that $a_r(\chi,v)=0$ for all $v \in V$.
Then by construction
\begin{equation*}\label{rec}
  u=\chi-c_0 \left({_0\hspace{-0.3mm}I^\al_x} (q\chi)\right) (1)(x^{\al-1}-x^2)
\end{equation*}
is a solution of \eqref{strongp} with the right hand side $f$. Meanwhile,
Assumption \ref{ass:riem} implies that problem \eqref{strongp} has a
unique solution
\begin{equation*}
  u=c(x^{\al-1}-x^2).
\end{equation*}
Then comparing these two solution representations yields that
\begin{equation}\label{chi}
  \chi=c'(x^{\al-1}-x^2).
\end{equation}
where $c'\neq0$. Now using the fact that
\begin{equation*}
A(x^{\al-1}-x^2,v) = (-\DDR 0 \alpha (x^{\al-1}-x^2),v)\quad \forall v\in V,
\end{equation*}
we deduce that $\chi$ satisfies \eqref{eqn:eqnforreg}.
Then plugging \eqref{chi} into \eqref{eqn:eqnforreg} and simple computation yield
\begin{equation*}
   c_0\ {_0\hspace{-0.3mm}I^\al_x} (q(x^{\al-1}-x^2)) (1) =1,
\end{equation*}
which is contradictory to the definition of $c_0$.
\end{proof}

Under Assumption \ref{ass:riem2}, the variational formulation \eqref{eqn:var2} is stable.
\begin{theorem}\label{thm:wpreg}
Let Assumption \ref{ass:riem2} hold and $q\in L^\infty\II$. Then for any $F\in V^*$, there
exists a unique solution $u\in V$ to
\begin{equation}\label{eqn:var3}
  a_r(u,v)= \langle F,v\rangle \quad \forall v\in V.
\end{equation}
\end{theorem}
\begin{proof}
The proof is based on Petree-Tartar Lemma \cite[pp. 469, Lemma A.38]{ern-guermond}.
We define two operators $S \in \mathcal{L}(V;V^*)$ and $T \in \mathcal{L}(V;V^*)$ by
\begin{equation*}
  \langle Su,v\rangle=a_r(u,v)\quad \text{and}\quad(Tu,v)=-b(u,v),
\end{equation*}
respectively. By Assumption \ref{ass:riem2}(a), the operator $S$ is injective. Further,
\begin{equation*}
  \begin{split}
   (Tu)(x)& = -\int_0^1 \frac{Q(x)q(y)(1-y)^{\al-1}}{\Ga(\al)}u(y)\,dy-q(x)u(x)\\
    &=: (T_1 u)(x) + (T_2u)(x).
  \end{split}
\end{equation*}
Now $Q\in L^\infty\II$ implies that $T_1$ is a Hilbert-Schmidt operator
and hence compact \cite[pp. 277, example 2]{Yoshida:1980}. 
Further, $T_2$ is also compact in view of the assumption $q\in L^\infty\II$ and the compactness of
the embedding from $V$ into $L^2\II$. Thus $T$ is a compact operator from $V$ to $L^2\II$.
By the coercivity \eqref{A-coercive} and continuity \eqref{continuous} of the
bilinear form $A(\cdot,\cdot)$, we obtain
\begin{equation*}
    C_0\|u\|_{V}^2 \le A(u,u)= a_r(u,u)-b(u,u)
    \le C \left(\| Tu \|_{V^\ast}  +  \| Su  \|_{V^*}\right) \| u \|_V,
\end{equation*}
Now the Petree-Tartar lemma immediately implies that the image of the operator $S$ is closed;
equivalently, there exists a constant $\delta>0$
satisfying
\begin{equation}\label{inf-suprl}
  \delta  \|u\|_V \le \sup_{v\in V } \frac {a_r(u,v)} {\|v\|_V}.
\end{equation}
This together with Assumption \ref{ass:riem2}$(\hbox{a}^\ast)$ yields the unique existence of
a solution $u\in U$ to the weak form \eqref{eqn:var3}.
\end{proof}

Now we state an improved regularity result for the case $\langle F,v\rangle
=(\tf,v)$, for some $\tf \in \Hd \gamma$, $0\leq \gamma\leq 1$, $\gamma\neq1/2$.
\begin{theorem}\label{thm:regrl-new}
Let Assumption \ref{ass:riem2} hold, $q\in \Hd{\ga}\cap L^\infty\II$ and $f\in \Hd{\ga}$
with $ 0\le \ga \le1$, $\gamma\neq 1/2$. There exists a unique solution $\ur\in \Hd {\al/2}$ to problem
\eqref{eqn:var2} and further, for any $\beta\in[1-\alpha/2,1/2)$, with $\ell(\beta,\gamma)=:\min(\al-1+\beta,\ga)$, it satisfies
\begin{equation*}
  \| \ur \|_{H^ {\alpha+\ell(\beta,\gamma)}\II} \le C \|f\|_{\Hd\ga}.
\end{equation*}
\end{theorem}
\begin{proof}
The unique existence of a solution $\ur \in V$ follows from
Theorem \ref{thm:wpreg}. Hence, it suffices to show the stability estimate.
By its construction, the solution $\ur $ is of the form \eqref{eqn:reg+sing}.
By $q \in \Hd {\ga}\cap L^\infty\II$ and $u\in \Hdi0 {\al-1+\beta}$,
and by Lemma \ref{lem:Hmulti}, we deduce $qu \in \Hdi0 {\min(\ga, \al-1+\beta)}$.
Now with Theorem \ref{thm:fracop}(c), ${_0\hspace{-0.3mm}I^\al_x}(f-qu)\in
\Hdi0{\alpha+\ga(\beta)}$.
\end{proof}

\begin{corollary}
Let Assumption \ref{ass:riem2} hold, $\ur$ be the solution of \eqref{eqn:var2}, and the
singularity strength $\mu$ be defined by \eqref{eqn:lambda}. Then $u=\ur+\mu(x^{\al-1}-x^2)$ is the solution of \eqref{strongp}.
\end{corollary}

\subsection{Adjoint problem}\label{secc:dual}
To derive error estimates for the Galerkin approximation $u_h^r$ of the regular part $u^r$ in Section \ref{sec:newFEM} below,
it is useful to consider the adjoint problem to \eqref{eqn:var2}. For $F \in V^*$, the dual
problem is to find $w\in V$ such that
\begin{equation}\label{eqn:dual}
    a_r(v,w)=\langle v,F \rangle \quad\forall v \in V.
\end{equation}
In the case of $\langle v,F \rangle = (v,f)$ for some $f \in L^2\II$, the strong form reads
\begin{equation}\label{eqn:dualstrong}
  \begin{aligned}
    -\DDR1 \al w + q  w  + \frac{(1-x)^{\al-1}q}{\Ga(\al)} (Q,w)&= f  \quad  \mbox{ in } D,\\
    w(0)=w(1)&=0.
  \end{aligned}
\end{equation}
For the case $q\equiv0$, the solution $w$ to \eqref{eqn:dualstrong} is given by
\begin{equation}\label{eqn:dualsol}
  w=-({_x\hspace{-0.3mm}I^\al_1}f)(x)+({_x\hspace{-0.3mm}I^\al_1}f)(0)(1-x)^{\al-1},
\end{equation}
and hence, $w\in \Hdi 1 {\al-1+\beta}\cap \Hd {\al/2}$. The case $q\neq0$ is treated
in the following theorem.
\begin{theorem}\label{thm:dualreg}
Let $q\in L^\infty\II$, and Assumption \ref{ass:riem2} hold. Then with a right hand side $\langle v,F \rangle = (v,f)$
for some $f \in L^2\II$, the solution $w$ to problem \eqref{eqn:dual} belongs to
$\Hdi1{\alpha-1+\beta} \cap \Hd{\alpha/2}$ and satisfies for any $\beta\in[1-\alpha/2,1/2)$
\begin{equation*}
  \|w\|_{\Hdi1 {\alpha-1+\beta}} \le C \|f\|_{L^2\II},
\end{equation*}
\end{theorem}
\begin{proof}
The strong problem for $w$ can be rewritten as
\begin{equation*}
    -\DDR1 \al w(x)=\bar f:=-  q(x) w(x) -
    \frac{(1-x)^{\al-1}q(x)}{\Ga(\al)} (Q,w) + f(x) ,\quad x \in D,
\end{equation*}
with the boundary condition $w(0)=w(1)=0$. Since $q\in L^\infty\II$
and $w\in \Hd {\al/2}$, there holds $qw\in L^2\II$, and the source term $\bar f
\in L^2\II$. Now the desired result follows directly from the representation \eqref{eqn:dualsol}
and Theorem \ref{thm:fracop}(c).
\end{proof}

\begin{remark}
In general, the best possible regularity of the solution $w$ to the adjoint problem \eqref{eqn:dual} lies in
$\Hdi1 {\al-1+\beta}$ for $\beta\in [1-\alpha/2,1)$ due to the presence of the singular term $(1-x)^{\al-1}$.
The only possibility of full regularity is the case $({_x\hspace{-0.3mm}I^\al_1}f)(0) = 0$ (for $q = 0$).
\end{remark}

\section{Galerkin FEM for the new formulation}\label{sec:newFEM}
Now we apply the variational formulation developed in Section \ref{sec:formulation} to the
numerical approximation of problem \eqref{strongp}. We shall analyze the stability of the
discrete variational formulation, and derive error estimates for the discrete approximations.
\subsection{Galerkin FEM}\label{ssec:FEM}
Based on the variational formulation \eqref{eqn:var2}, we can develop a new Galerkin FEM for problem
\eqref{strongp} with enhanced convergence rates. First, we approximate the regular part $\ur$ using a
Galerkin procedure over the continuous piecewise linear finite element space $V_h$. The choice of
piecewise linear elements is motivated by the following empirical observation: due to the presence of the
potential term $q$, the solution $u$ to the fractional model \eqref{strongp}
can contains a hierarchy of weak singularities, apart from the leading one $x^{\alpha-1}$, and thus
higher-order elements are not expected to be efficient in general.
The discrete counterpart of \eqref{eqn:var2} is to find $\ur_h\in V_h$ such that
\begin{equation}\label{eqn:fem}
    a_r(\ur_h,v)=(\tf,v) \quad \forall v \in V_h,
\end{equation}
where the bilinear form $a_r(\cdot,\cdot)$ and the source term $\tf\in L^2\II$ are defined in
\eqref{eqn:var2} and \eqref{eqn:Q+tf}, respectively. Then we reconstruct
a finite element approximation $\mu_h$ to the strength $\mu$ of the singular part $u^s$ by
\begin{equation}\label{eqn:lambdah}
  \mu_h=c_0 \left({_0\hspace{-0.3mm}I^\al_x} (f-q\ur_h)\right) (1),
\end{equation}
where the constant $c_0$  is defined in \eqref{eqn:c0}. Last, we construct an approximate
solution $u_h$ to \eqref{strongp} by
\begin{equation}\label{eqn:uh}
  u_h=\ur_h+\mu_h(x^{\al-1}-x^2).
\end{equation}
In order to derive an error estimate, we first establish the well-posedness of
problem \eqref{eqn:fem}. To this end, we need the (adjoint) Ritz projection $R_h:
V\rightarrow V_h$ defined by
\begin{equation}\label{eqn:Ritz}
    A(v,R_hu)=A(v,u) \quad \forall u\in V,  v\in V_h.
\end{equation}
Then C\'{e}a' lemma and finite element duality imply that there hold for any $\beta\in[1-\alpha/2,1/2)$
\begin{equation}\label{eqn:Ritzfrac}
  \begin{aligned}
    \|R_h u\|_{V} & \le C \| u\|_{V} &&  \forall u \in V,\\
    \| u- R_h u \|_{L^{2}\II} & \le C  h^{\al/2 -1 + \beta}  \| u \|_{V} && \forall u \in V.
  \end{aligned}
\end{equation}
Note that the $L^2\II$ error estimates of the adjoint Ritz projection $R_h$ is suboptimal,
due to the low global regularity of the adjoint solution.

Next, we show the stability of the discrete variational problem \eqref{eqn:fem}.
\begin{theorem}
Let Assumption \ref{ass:riem2} hold, $f\in L^2\II$, and $q\in L^\infty\II$. Then there is
an $h_0>0$ such that for all $h\le h_0$ the finite element problem: finding $\ur_h\in V_h$ such that
\begin{equation}\label{discp}
a_r(\ur_h,v)=(f,v)  \quad \forall v\in V_h
\end{equation}
has a unique solution.
\end{theorem}
\begin{proof}
The unique existence of the discrete solution $\ur_h\in V_h$ when $q\equiv0$ is a
direct consequence of the coercivity and continuity of the bilinear form $A(\cdot,\cdot)$
on the space $V\times V$, cf. \eqref{A-coercive} and \eqref{continuous}, and Lax-Milgram
theorem. For the case $q\neq0$, we show the inf-sup condition for the bilinear form
$a_r(\cdot,\cdot)$ on the space $V_h\times V_h$ using a kick-back argument analogous to
Schatz \cite{Schatz-1974}.

By the {{inf-sup}} condition \eqref{inf-suprl},
we have for $z_h\in V_h \subset V\equiv\Hd{\al/2}$
\begin{equation}\label{eqn:discstabest1}
   \delta \| z_h \|_{V} \le \sup_{v\in V} \frac{a_r(z_h,v)}{\|v\|_{V}}
   \le \sup_{v\in V} \frac{a_r(z_h,v-R_hv)}{\|v\|_{V}}
   +\sup_{v\in V} \frac{a_r(z_h,R_hv)}{\|v\|_{V}}
\end{equation}
In view of Galerkin orthogonality of (adjoint) Ritz-projection $R_h$
and the continuity of bilinear form $b(\cdot,\cdot)$ in \eqref{eqn:bconti}, we have
\begin{equation*}
\begin{split}
   \sup_{v\in V} \frac{a_r(z_h,v-R_hv)}{\|v\|_V}
   = \sup_{v\in V} \frac{b(z_h,v-R_hv)}{\| v\|_V}
   \le C \sup_{v\in V} \frac{\|z_h\|_{L^2\II}\|v-R_hv\|_{L^2\II}}{\|v\|_V}.
\end{split}
\end{equation*}
For the first term, using \eqref{eqn:Ritzfrac}, we obtain for all $\beta\in [1-\alpha/2,1/2)$
\begin{equation*}
   \sup_{v\in V} \frac{a_r(z_h,v-R_hv)}{\|v\|_V}
   \le Ch^{\al/2-1+\beta}\|z_h\|_{L^2\II}\le C_1h^{\al/2-1+\beta}\|z_h\|_V .
\end{equation*}
Now by the stability of $R_h$ in \eqref{eqn:Ritzfrac},
we deduce the following estimate for the second term
\begin{equation*}
   \sup_{v\in V} \frac{a_r(z_h,R_hv)}{\| v \|_V}
\le C\sup_{v\in V} \frac{a_r(z_h,R_hv)}{\| R_hv \|_V}
\le C_2\sup_{v\in V_h}\frac{a_r(z_h,v)}{\|v\|_V}.
\end{equation*}
By choosing $h_0$ such that $C_1 h_0^{\al/2-1+\beta} = \delta/2$,
we arrive at the desired discrete inf-sup condition on $V_h\times V_h$:
\begin{equation}\label{discr-inf-sup}
  \frac{\delta}{2} \|z_h\|_V \le C_2 \sup_{v \in V_h} \frac{a_r(z_h,v)}{ \|v\|_V}
  \quad \text{for} \quad h \le h_0.
\end{equation}
This shows that there is a unique solution of \eqref{eqn:fem}.
\end{proof}

Before providing error estimates, we note that the matrix analogue of the bilinear form $A(\cdot,\cdot)$ is
of lower-Hessenberg form, and that for the integral part is of rank one. Hence,
the matrix for the bilinear form $a_r(\cdot,\cdot)$ is a rank-one perturbation of
a lower-Hessenberg matrix. Further, the stiffness matrix for the leading term $-(\DDR0{\alpha/2}\cdot,\
\DDR1{\alpha/2}\cdot)$ is Toeplitz when the mesh is uniform, which can be verified directly. 
This represents a very important structural property that can be exploited in forming the
stiffness matrix, efficient storage and fast iterative solution of the resulting linear system, by e.g., GMRES.

\subsection{Error estimates}\label{ssec:error}
Next we establish error estimates for the Galerkin approximation $u_h$ given by \eqref{eqn:uh}. We first
analyze the $L^2\II$- and $\widetilde H^{\al/2}\II$-norm of the error $\ur-\ur_h$ for the regular part $\ur$.
\begin{theorem}\label{thm:regerror}
Let Assumption \ref{ass:riem2} hold, $f\in \Hd {\ga} $, and $ q \in \Hd{\ga} \cap L^\infty\II$, $0\leq\gamma\leq
1$, $\gamma\neq 1/2$. Then there is an $h_0>0$ such that for all $h\le h_0$, the solution $\ur_h$ to problem
\eqref{eqn:fem} satisfies for any $\beta\in[1-\alpha/2,1/2)$, with $\ell(\beta,\gamma)=:\min(\al-1+\beta,\ga)$
\begin{equation*}
   \|\ur-\ur_h\|_{L^2\II} + h^{\al/2-1+\beta}\|\ur-\ur_h\|_{\Hd{\alpha/2}}
   \le C h^{\min(2,\al+\ell(\beta,\gamma))-1+\beta} \|f\|_{\Hd\ga}.
\end{equation*}
\end{theorem}
\begin{proof}
The error estimate in the $\Hd{\al/2}$-norm follows directly from C\'{e}a's lemma,
\eqref{discr-inf-sup} and the Galerkin orthogonality. Specifically, for all $h\le h_0$ and
any $ \chi \in V_h$ we have by Galerkin orthogonality, \eqref{eqn:arcontinuous} and \eqref{discr-inf-sup}
\begin{equation*}
\begin{split}
    \frac \delta 2\| \ur_h-\chi \|_{V}
    &\le C\sup_{v \in V_h} \frac{a_r(\ur_h-\chi,v)}{ \|v\|_V}\\
    &\le C\sup_{v \in V_h} \frac{a_r(\ur-\chi,v)}{ \|v\|_V}
    \le C\| \ur-\chi \|_{V}.
\end{split}
\end{equation*}
Hence the triangle inequality yields for all $\chi \in V_h$
\begin{equation*}
\begin{split}
    \| \ur-\ur_h \|_{V} \le \|  \ur-\chi \|_V + \|\chi-\ur_h \|_V
    \le C \| \ur-\chi \|_V.
\end{split}
\end{equation*}
Then the desired $\Hd{\alpha/2}$-estimate follows from Lemma \ref{fem-interp-U} by
\begin{equation*}
\begin{split}
    \| \ur-\ur_h \|_V &\le C\inf_{\chi \in V_h} \| \ur -\chi \|_V \\
    &\le C h^{\min(2,\al+\ell(\beta,\gamma))-\al/2}\| \ur \|_{\al+\ga(\beta)}\\
    &\le C h^{\min(2,\al+\ell(\beta,\gamma))-\al/2} \| f \|_{\Hd{\ga}}.
\end{split}
\end{equation*}
Then we apply Nitsche's trick to establish the $L^2\II$-error estimate.
To this end, we consider the adjoint problem \eqref{eqn:dual} with $f=\ur-\ur_h$, i.e.
\begin{equation*}
\begin{split}
    \| \ur-\ur_h \|_{L^2\II}^2 = a_r(\ur-\ur_h, w)= a_r(\ur-\ur_h,w-w_h),
\end{split}
\end{equation*}
for any $w_h\in V_h$. 
Then Lemma \ref{fem-interp-U}, Theorem \ref{thm:dualreg} and \eqref{eqn:arcontinuous} yield
for any $\beta\in[1-\alpha/2,1/2)$
\begin{equation*}
\begin{split}
    \| \ur-\ur_h \|_{L^2\II}^2 &\le \|\ur-\ur_h\|_{V}\inf_{w_h\in V_h}\|w-w_h\|_{V}\\
    &\le C h^{\min(2,\al+\ell(\beta,\gamma))-1+\beta} \| f \|_{\Hd{\ga}} \| \ur-\ur_h \|_{L^2\II}.
\end{split}
\end{equation*}
This completes the proof of the theorem.
\end{proof}

Now we turn to the reconstruction $\mu_h$ of the singularity strength $\mu$.
\begin{lemma}\label{lem:singerror}
Let the assumptions in Theorem \ref{thm:regerror} hold. Then there is an $h_0>0$ such
that for all $h\le h_0$, the solution $\mu_h$ satisfies that for any $\beta\in[1-\alpha/2,1/2)$, with
$\ell(\beta,\gamma)=:\min(\al-1+\beta,\ga)$
\begin{equation*}
   |\mu-\mu_h|\le C h^{\min(2,\al+\ell(\beta,\gamma))-1+\beta} \|f\|_{\Hd \ga}.
\end{equation*}
\end{lemma}
\begin{proof}
We first recall that $q(\ur-\ur_h)\in L^2\II$ since $q\in L^\infty\II$ and $\ur-\ur_h\in L^2\II$.
Thus by Sobolev imbedding theorem, we have
\begin{equation}
    |\mu-\mu_h|=|c_0|\big|{_0\hspace{-0.3mm}I^\al_x}(q(\ur-\ur_h))(1)\big|
    \le C\|{_0\hspace{-0.3mm}I^\al_x}(q(\ur-\ur_h))\|_{\Hdi0{\al}}.
\end{equation}
Then by Theorems \ref{thm:fracop}(c) and \ref{thm:regerror},
we have for any $\beta\in[1-\alpha/2,1/2)$
\begin{equation}
    |\mu-\mu_h| \le C \| \ur-\ur_h\|_{L^2\II} \le C h^{\min(2,\al+\ell(\beta,\gamma))-1+\beta} \|f\|_{\Hd \ga}.
\end{equation}
\end{proof}

Now we can derive a first error estimate for the approximation $u_h$ defined in \eqref{eqn:uh}.
\begin{theorem}\label{lem:femerror}
Let the assumptions in Theorem \ref{thm:regerror} hold. Then there is an $h_0>0$ such
that for all $h\le h_0$, the solution $u_h$ satisfies that for any $\beta\in[1-\alpha/2,1/2)$, with
$\ell(\beta,\gamma)=:\min(\al-1+\beta,\ga)$
\begin{equation*}
   \|u-u_h\|_{L^2\II} + h^{\al/2-1+\beta}\|u-u_h\|_{\Hd{\alpha/2}}
   \le C h^{\min(2,\al+\ell(\beta,\gamma))-1+\beta} \|f\|_{\Hd\ga}.
\end{equation*}
\end{theorem}
\begin{proof}
The definitions of $u$ and $u_h$ imply
\begin{equation*}
  \|u-u_h\|_{\Hd{\alpha/2}}\le \|\ur-\ur_h\|_{\Hd{\alpha/2}}+|\mu-\mu_h|\|x^{\al-1}-x^2\|_{\Hd{\alpha/2}}.
\end{equation*}
Then the desired estimate follows from Theorem \ref{thm:regerror} and Lemma
\ref{lem:singerror}. The proof of the $L^2\II$-estimate is analogously and hence omitted.
\end{proof}

\begin{remark}
By Theorem \ref{lem:femerror}, for $q\in L^\infty\II$ and $f\in L^2\II$, the $L^2\II$- and $\Hd{\al/2}$-norm of
the error can be respectively almost of order $ O(h^{\al-1/2})$ and $O(h^{\al/2})$, even though the solution $u$
is only in $\Hdi 0 {\al-1+\beta}$ for $\beta\in[1-\alpha/2,1/2)$, which is better than that for the standard Galerkin
method by $O(h^{\alpha/2})$ for both the $L^2\II$ and $\Hd {\alpha/2}$ error estimates, cf. Theorem \ref{thm:femrl}.
Hence, in comparison with the standard Galerkin FEM, the new approach does yield a higher-order convergence rate.
\end{remark}

\begin{remark}
In contrast to the standard Galerkin FEM, the convergence rate of the singularity reconstruction
technique can be further picked up, if $q$ and $f$ are smoother. Specifically, for sufficiently smooth
$q$ and $f$, the $L^2\II$ and $\Hd{\al/2}$-estimates are respectively almost $ O(h^{\min(3/2,2\al-1)})$
and $O(h^{\min(2-\al/2,3\al/2-1/2)})$, even though the regularity of the solution $u$ remains only in
$\Hdi0{\al-1+\beta}\cap \Hd {\al/2}$. The $L^2\II$ estimate is not sharp, due to the low global regularity
of the adjoint problem. However, we note that the adjoint problem has its leading singularity concentrated only at
one point, i.e., $x=1$. This structure is not exploited in Nitsche's argument. It is still unclear how
to incorporate the structure.  One possibility is first to develop local error estimates.
\end{remark}

Last we shall derive an $L^\infty\II$-error estimate for the approximation $u_h$. To
this end, we consider Green's function for problem
\eqref{eqn:dualstrong}, i.e., to find $G(x,y)$ such that for any fixed $x\in D$
\begin{equation}\label{eqn:newgreen}
        -{\DDR1 \al} G(x,y) + q(y) G(x,y) +
    \frac{(1-y)^{\al-1}q(y)}{\Ga(\al)} (Q,G(x,\cdot)) = \delta_x(y) , \quad y \in D,
\end{equation}
with the boundary condition $G(x,0)=G(x,1)=0$. We note that the
variational formulation of Green's function is given by:  find
$G(x,y)\in V\equiv\Hd{\al/2}$ such that for all $x \in D$
\begin{equation}\label{eqn:newgreenweak}
     a_r(v,G(x,\cdot))=\langle v,\delta_x \rangle=v(x) \quad \forall  v\in V.
\end{equation}
The existence and uniqueness of Green's function $G(x,y) \in V$ follows directly
from Theorem \ref{thm:dualreg} and the fact that $\delta_x\in V^*$. Further, when
 $q\equiv0$ the problem reduces to the standard fractional boundary value
problem, for which the Green's function has the explicit representation \eqref{eqn:greenq0}
and $G(x,\cdot)\in \Hdi1{\al-1+\beta}$ with $\beta \in [1-\alpha/2,1/2)$. In case of a general
$q\neq 0$, the Green's function $G(x,y)$ to problem \eqref{eqn:newgreen} satisfies
\begin{equation*}
-\DDR1 \al G(x,y) = \delta_x(y) - q(y) G(x,y)
    - \frac{(1-y)^{\al-1}q(y)}{\Ga(\al)} (Q,G(x,\cdot)) .
\end{equation*}
Hence, for $q\in L^\infty\II$, $q(\cdot)G(x,\cdot)\in L^2\II$ and ${_yI_1^{\al}}\delta_x \in \Hdi1{\al-1+\beta}$,
and thus Green's function $G(x,\cdot)$ belongs to $\Hdi1{\al-1+\beta}$ with $\beta\in[1-\alpha/2,1/2)$.

Now we can state an error estimate in the $L^\infty\II$-norm for the approximation $u_h$.
\begin{theorem}\label{thm:femerrorinf}
Let the assumptions in Theorem \ref{thm:regerror} hold. Then there is an $h_0>0$ such
that for all $h\le h_0$, the solution $u_h$ satisfies that for any $\beta\in[1-\alpha/2,1/2)$,
with $\ell(\beta,\gamma)=:\min(\al-1+\beta,\ga)$, there holds
\begin{equation*}
   \|u-u_h\|_{L^\infty \II} \le C h^{\min(2,\al+\ell(\beta,\gamma))-1+\beta} \|f\|_{\Hd \ga}.
\end{equation*}
\end{theorem}
\begin{proof}
Like before, we first derive an $L^\infty\II$-estimate of $\ur-\ur_h$.
By the weak formulation of Green's function \eqref{eqn:newgreenweak}
and Galerkin orthogonality, we have for all $x\in D$ and $w_h\in V_h$
\begin{equation*}
\begin{split}
(\ur-\ur_h)(x)&= a_r(\ur-\ur_h, G(x,\cdot))=a_r(\ur-\ur_h, G(x,\cdot)-w_h).
\end{split}
\end{equation*}
Then by \eqref{continuous}, Lemma \ref{fem-interp-U} and Theorem \ref{thm:regrl-new} we obtain for any $\beta\in[1-\alpha/2,1/2)$
\begin{equation*}
\begin{split}
|(\ur-\ur_h)(x)|&\le \| \ur-\ur_h\|_V \inf_{w_h\in V_h}\|G(x,\cdot)-w_h\|_V \\
  &\le Ch^{\min(2,\al+\ell(\beta,\gamma))-1+\beta}\| \ur \|_{\Hdi0{\al+\ell(\beta,\gamma)}}\\
  &\le Ch^{\min(2,\al+\ell(\beta,\gamma))-1+\beta}\| f \|_{\Hd\ga}.
\end{split}
\end{equation*}
Then Lemma \ref{lem:singerror} yields
\begin{equation*}
\begin{split}
\|u-u_h\|_{L^\infty\II} &\le \| \ur-\ur_h \|_{L^\infty\II}+|\mu-\mu_h|\| x^{\al-1}-x^2 \|_{L^\infty\II}\\
    &\le Ch^{\min(2,\al+\ell(\beta,\gamma))-1+\beta}\| f \|_{\Hd\ga}.
\end{split}
\end{equation*}
This completes the proof of the theorem.
\end{proof}

\begin{remark}
In view of Theorems \ref{thm:linf} and \ref{thm:femerrorinf}, for $f\in L^2\II$, the $L^\infty\II$-norm of the
error is one half order higher for the new strategy than the standard Galerkin FEM. If the source term $f$ is smoother (in the space $\Hdi 0 s$),
the new technique yields an even higher uniform convergence rate due to the enhance regularity of the regular
part $\ur$. Further, numerically, a superconvergence phenomenon in the $L^\infty\II$-norm is observed for the
case $\al+\ell(\beta,\gamma)>2$.
\end{remark}

\section{Extension to mixed boundary value problem}\label{sec:neumann}
The reconstruction technique is very versatile, and it can be straightforwardly
extended to other type boundary conditions. We illustrate it with the following boundary
value problem with a mixed boundary condition: find $u\in L^2\II$ such that
\begin{equation}\label{strongp1}
  \begin{aligned}
   -{\DDR 0 \alpha} u + qu = f & \quad  \mbox{ in } D\\
   \DDR0 {\al-1} u(0)=u(1)=0, &
   \end{aligned}
\end{equation}
with $\alpha\in(3/2,2)$. The choice $\al\in(3/2,2)$ is to ensure that problem \eqref{strongp1} has a solution $u$ in
$L^2\II$; see the discussions below. To the best of our knowledge, the case of a mixed boundary condition
like problem \eqref{strongp1} has not been analyzed in the literature.

\subsection{Well-posedness}
We first discuss the well-posedness of problem \eqref{strongp1} and the regularity pickup. Let
$g={_0\hspace{-0.3mm}I^\al_x} f$. By \cite[Section 3]{JinLazarovPasciak:2013a} (also \cite[pp. 44,
Theorem 2.4]{SamkoKilbasMarichev:1993}), we have $\DDR 0 \alpha g = f$. This together with the
identities $(\DDR 0 {\alpha-1} g) (0) =0$ and $\DDR 0 {\alpha -1} x^{\alpha-2} = 0$ implies
that in case of $q\equiv0$, problem \eqref{strongp1} has a solution
\begin{equation}\label{eqn:sol1}
  u(x) =- ({_0\hspace{-0.3mm}I^\al_x} f)(x) +({_0\hspace{-0.3mm}I^\al_x} f)(1) x^{\al-2},
\end{equation}
which belongs to $\Hdi0{\al-3/2-\ep}\subset L^2\II$, $\epsilon\in (0,\alpha-3/2)$. To analyze the 
well-posedness of the problem, we introduce the following
function spaces
\begin{equation}\label{eqn:spaceU}
  U=\left(\Hdi0{\al-1}\oplus\{ x^{\al-2}\}\right)\cap\{u(1)=0\}
  \quad \mbox{and}\quad V=\Hdi1 1.
\end{equation}
Clearly, any element $u\in U$ can be uniquely represented by $u=u_0-u_0(1)x^{\al-2}$ for some $u_0 \in\Hdi0{\al-1}  $.
The next lemma shows that the functional
\begin{equation}\label{eqn:norm-U}
 \| u  \|_U = \|  \DDR0 {\al-1}u_0  \|_{L^2\II}
\end{equation}
defines a norm in the space $U$.

\begin{lemma}\label{lem:normU}
For $u\in U$, define the functional by \eqref{eqn:norm-U}.
Then $U$ is a Banach space with the norm $\| \cdot \|_U$.
\end{lemma}
\begin{proof}
The triangle inequality and absolute homogeneity follow immediately. It suffices to show that $\| u \|_U$ implies $u=0$ for all
$u\in U$. By Theorem \ref{thm:fracop}(b), we have
\begin{equation*}
 \| \DDR0{\al-1}\fy \|_{L^2\II} \le C \| \fy  \|_{\Hdi0{\al-1}}, \quad \forall \fy \in \Hdi0{\al-1}.
\end{equation*}
Meanwhile, for $\fy \in \Hdi 0 {\alpha-1}$, by setting $ v= {\DDR0{\al-1}\fy}$, we deduce
$\fy={_0\hspace{-0.3mm}I^{\al-1}_x}v(x)$, and by Theorem \ref{thm:fracop}(c), there holds
\begin{equation*}
 \| \fy  \|_{\Hdi0{\al-1}} = \| {_0\hspace{-0.3mm}I^{\al-1}_x}v  \|_{\Hdi0{\al-1}}\le C\| v \|_{L^2\II}= C\| \DDR0{\al-1}\fy \|_{L^2\II}.
\end{equation*}
Thus the seminorm $\| \DDR0{\al-1} \cdot \|_{L^2\II}$ is equivalent to the norm $\| \cdot \|_{\Hdi0{\al-1}}$ in
the space $\Hdi0{\al-1}$. Then we deduce that
\begin{equation*}
\| u \|_U = \|  \DDR0{\al-1} u_0 \|_{L^2\II}=0 \Rightarrow  \| u_0 \|_{\Hdi0{\al-1}}=0  \Rightarrow u_0 =0 \Rightarrow u=0.
\end{equation*}
Further, the completeness of the space follows from the completeness of $\Hdi0{\al-1}$ and Sobolev imbedding theorem.
\end{proof}

Next we establish a stable variational formulation and derive the regularity pickup. We
define a bilinear form $A(\cdot,\cdot):U\times V \rightarrow \mathbb{R}$ by
\begin{equation*}
A(u,\fy)=(\DDR0 {\al-1}u,~~\fy').
\end{equation*}
It can be verified directly that the representation \eqref{eqn:sol1} satisfies the following Petrov-Galerkin formulation
\begin{equation}\label{eqn:weakform1}
 A(u,v)= (f,v)\quad \forall v\in V.
\end{equation}
Now we show the inf-sup condition of the bilinear form $A(\cdot,\cdot)$. For any fixed $u\in U$,
by choosing $v_u={_0\hspace{-0.3mm}I^{2-\al}_x}u - ({_0\hspace{-0.3mm}I^{2-\al}_x}u)(1)\in \Hdi 1 1 $
we obtain
\begin{equation*}
 \sup_{v \in V}  \frac{A(u,v)}{\| v' \|_{L^2\II}} \ge \frac{A(u,v_u)}{\| v_u' \|_{L^2\II}} = \| u \|_U.
\end{equation*}
Now for any nonzero $v \in V$, by choosing $u_v={_0\hspace{-0.3mm}I^{\al-1}_x}v' - ({_0\hspace{-0.3mm}I^{\al-1}_x}v')(1)x^{\al-2}\in U$
we have
\begin{equation*}
 A(u_v,v) = \| v'  \|_{L^2\II} > 0,
\end{equation*}
which implies that the inf-sup condition of the adjoint problem holds as well. Consequently, problem \eqref{eqn:weakform1}
with $q\equiv0$ has a unique solution in the space $U$, and it can be represented by \eqref{eqn:sol1}.

In the general case $q\neq 0$, we define
\begin{equation*}
a(u,v) = A(u,v) +(qu,v).
\end{equation*}
Then the variational formulation for problem \eqref{strongp1} is given
by: find $u\in U$ such that
\begin{equation}\label{varr-nm}
    a(u,v)=(f,v)\quad \forall v\in V.
\end{equation}
To study the bilinear form $a(\cdot,\cdot)$, we make the following assumption analogous to Assumption \ref{ass:riem}.
\begin{assumption} \label{ass:riem-nm}
Let the bilinear form $a(u,v)$ with $u\in U$ and $v\in V$ satisfy
\begin{itemize}
 \item[{$\mathrm{(a)}$}]  The problem of finding $u \in U$ such that $a(u,v)=0$ for all $v \in V$
           has only the trivial solution $u\equiv 0$.
 \item[{$(\mathrm{a}^\ast)$}] The problem of finding $v \in V$ such that $a(u,v)=0$ for all $u \in U$
    has only the trivial solution $v\equiv 0$.
\end{itemize}
\end{assumption}

Under Assumption \ref{ass:riem-nm}, we have the following theorem.
\begin{theorem}\label{thm:reg-nm}
Let $\al\in(3/2,2)$, $q\in \Hd \gamma$, $f\in \Hdi 0 \gamma$, $0\leq \gamma\leq 1$, $\gamma\neq1/2$, and Assumption
\ref{ass:riem-nm} hold. Then there exists a unique solution $u=u_0-u_0(1)x^{\al-2}\in U$ with $u_0 \in \Hdi0{\al+\beta}$
to problem \eqref{varr-nm} and further, for any $\beta\in[2-\alpha,1/2)$, with $\ell_n(\beta,\gamma)=:\min(\al-2+\beta,\ga)$,
there holds
\begin{equation*}
 \| u_0 \|_{\Hdi0{\al+\ell_n(\beta,\gamma)}} \le C \|  f \|_{\Hdi0{\gamma}}.
\end{equation*}
\end{theorem}
\begin{proof}
Under Assumption \ref{ass:riem-nm}, the proof of uniqueness and existence of a solution $u\in U$ is similar
to that of Theorem \ref{thm:wpreg}, and hence omitted. By Lemma \ref{lem:Hmulti}, the function $qu \in \Hd
{\min(\al-2+\beta,\gamma)}$. Then the regularity of  $u_0$ follows from the representation $u_0(x)= -(
{_0\hspace{-0.3mm}I^{\al}_x} (f-qu))(x)$ and Theorem \ref{thm:fracop}(c).
\end{proof}

\subsection{New variational formulation and FE approximation}
A direct application of the variational formulation \eqref{varr-nm} is inefficient for the numerical solution
of problem \eqref{strongp1}, due to the low solution regularity, as a consequence of the presence of the term $x^{\alpha-2}$.
To enhance the efficiency, we employ the singularity reconstruction technique, and consider the following splitting
\begin{equation}\label{splitting-nm}
 u= u^r + \mu (x^{\al-2}-x^2).
\end{equation}
Repeating the arguments in Section \ref{sec:formulation} yields the following integro-differential problem for the regular part $\ur$
\begin{equation}\label{eqn:eqnforreg-nm}
  \begin{aligned}
    -\DDR 0 \alpha \ur + q\ur + \left({_0\hspace{-0.3mm}I^\al_x}(q\ur)\right)& (1)Q  = \tf  \quad  \mbox{ in } D,\\
   \ur(0)=\ur(1)=0. &
   \end{aligned}
\end{equation}
It is worth noting that problem \eqref{eqn:eqnforreg-nm} has a homogeneous Dirichlet boundary condition
like in Section \ref{sec:formulation}. In problem \eqref{eqn:eqnforreg-nm}, the functions $Q(x)$ and
$\tf(x)$ are given respectively by
\begin{equation}
  \begin{aligned}\label{eqn:Q+tf-nm}
     Q(x)&=c_0c_1(x)-c_0q(x)(x^{\al-2}-x^2)\in L^{2}\II,\\
    \tf(x)&=f(x)+c_0c_1(x)\left({_0\hspace{-0.3mm}I^\al_x}f\right)(1)-
       c_0\left({_0\hspace{-0.3mm}I^\al_x}f\right)(1)q(x)(x^{\al-2}-x^2)\in L^2\II.
  \end{aligned}
\end{equation}
where the constant $c_0$ and the function $c_1(x)$ are respectively defined by
\begin{equation}\label{eqn:c0-nm}
  c_0=\frac{1}{1+{_0\hspace{-0.3mm}I^\al_x} (q(x^{\al-2}-x^2)) (1)}
  \quad \text{and}\quad c_1= {\DDR0{\al} (x^{\al-2}-x^{2})}=-\frac{2}{\Gamma(3-\al)}x^{2-\al}.
\end{equation}

The preceding discussions indicate that the singularity reconstruction technique handles the mixed problem
in the same manner as for the Dirichlet problem, with their only difference lying in replacing the function
$x^{\alpha-1}$ in the Dirichlet case with $x^{\alpha-2}$ in the mixed case. This shows clearly its versatility.
Hence, the corresponding bilinear form $a_r(\cdot,\cdot)$ is given by
\begin{equation*}\label{eqn:ar-nm}
    a_r(u,v)= A(u,v)+(qu,v)+ {_0\hspace{-0.3mm}I^\al_x}(qu)(1)(Q,v)\quad \forall u,v\in V\equiv\Hd{\al/2},
\end{equation*}
where the bilinear form $a_r(\cdot,\cdot)$ is defined in \eqref{eqn:ar}, and the function $Q$ is defined in \eqref{eqn:Q+tf-nm}.
Then the variational problem for the regular part $\ur$ reads: find $\ur \in V$ satisfying
\begin{equation}\label{eqn:var2-nm}
 a_r(\ur,v)= (\tf,v) \quad \forall v\in V.
\end{equation}
Once the regular part $u^r$ is determined, the singularity strength $\mu$ can be obtained by
\begin{equation}\label{eqn:mu-nm}
  \mu=c_0 \left({_0\hspace{-0.3mm}I^\al_x} (f-qu^r)\right) (1),
\end{equation}
Finally the solution $u$ of the source problem \eqref{strongp1} is recovered by \eqref{splitting-nm}.
The variational formulation \eqref{eqn:var2-nm} lends itself to the following discrete problem:
find $\ur_h\in V_h$ such that
\begin{equation}\label{eqn:fem-nm}
    a_r(\ur_h,v)=(\tf,v) \quad \forall v \in V_h,
\end{equation}
where the bilinear form $a_r(\cdot,\cdot)$ and the source term $\tf\in L^2\II$ are defined in
\eqref{eqn:var2-nm} and \eqref{eqn:Q+tf-nm}, respectively. Then we construct
a finite element approximation $\mu_h$ to the strength $\mu$ of the singular part $u^s$ by
\begin{equation}\label{eqn:muh-nm}
  \mu_h=c_0 \left({_0\hspace{-0.3mm}I^\al_x} (f-q\ur_h)\right) (1),
\end{equation}
where the constant $c_0$  is defined in \eqref{eqn:c0-nm}. Last, we construct an approximate
solution $u_h$ to \eqref{strongp1} by
\begin{equation}\label{eqn:uh-nm}
  u_h=\ur_h+\mu_h(x^{\al-2}-x^2).
\end{equation}

We have the following $L^\infty\II$, $L^2\II$ and $\Hd {\al/2}$-norm error estimates on the
Galerkin approximation $u_h^r$ and the singularity strength $\mu_h$. The proof is omitted since
it is identical to that in Section \ref{sec:newFEM}.
\begin{theorem}\label{thm:regerror-Neumann}
Let Assumption \ref{ass:riem2} hold, $f\in \Hd {\ga} $, and $ q \in \Hd{\ga} \cap L^\infty\II$, $0\leq \gamma\leq1$, $\gamma\neq1/2$.
Then there is an $h_0>0$ such that for all $h\le h_0$, the solution $\ur_h$ to problem
\eqref{eqn:fem-nm} satisfies for any $\beta\in[1-\alpha/2,1/2)$, with $\ell_n(\beta,\gamma)=:\min(\al-2+\beta,\ga)$
\begin{equation*}
\begin{split}
   \|\ur-\ur_h\|_{L^2\II} + h^{\al/2-1+\beta}&\|\ur-\ur_h\|_{\Hd{\alpha/2}}
   \le C h^{\min(2,\al+\ell_n(\beta,\gamma))-1+\beta} \|f\|_{\Hd\ga}.\\
\end{split}
\end{equation*}
Further, let $\mu$ and $\mu_h$ be defined in \eqref{eqn:mu-nm} and \eqref{eqn:muh-nm}. Then there holds
\begin{equation*}
 |\mu-\mu_h| \le  C h^{\min(2,\al+\ell_n(\beta,\gamma))-1+\beta} \|f\|_{\Hd\ga}.
\end{equation*}
\end{theorem}

\begin{remark}
The $L^2\II$ convergence rate of the approximation $u_h$ defined in \eqref{eqn:uh-nm} follows directly from Theorem
\ref{thm:regerror-Neumann}. Like the Dirichlet case, the $L^2\II$ error estimate is suboptimal, due to the limited
regularity of the adjoint solution. However, the $L^\infty\II$ and $\Hdi 1 {\alpha/2}$ error estimates do not follow, since
the exact solution $u$ is generally neither bounded nor in $\Hdi 1 {\alpha/2}$, due to the presence of the term $x^{\alpha-2}$
in the solution representation.
\end{remark}
\section{Numerical results and discussions}\label{sec:numeric}
In this section we present numerical experiments to verify our theoretical findings. We consider
the following three different source terms:
\begin{itemize}
  \item[(a)] The source term $f(x)=x(1-x)$ belongs to the space $\Hd {1+\ep}$ for any $\ep\in [0,1/2)$.
  \item[(b)] The source term $f(x)=\chi_{[0,1/2]}$ belongs to the space $\Hd {\ep}$ for any $\ep\in [0,1/2)$.
  \item[(c)] The source term $f(x)=x^{-1/4}$ belongs to the space $\Hd {\ep}$ for any $\ep\in [0,1/4)$.
\end{itemize}

The computations were performed on uniform meshes of mesh sizes $h=1/2^k$, $k=5,6,\ldots,10$.
We note that if the potential $q$ is zero, the exact solution $u$ can be computed directly using the
solution representation \eqref{eqn:sols}, and similarly the regular part $\ur$ can be evaluated in
closed form. For the case $q\neq 0$, the exact solution is not available explicitly, and hence
we compute the reference solution using a very refined mesh with a mesh size $h= 1/2^{13}$. For each example,
we consider three different $\alpha$ values, i.e., $5/4$, $3/2$ and $7/4$, and present the
$L^2\II$, $\Hd {\al/2}$, and $L^\infty\II$-norm of the error $e=u^r-u_h^r$ of the regular
part and the error $|\mu-\mu_h|$ of the singularity strength $\mu$ separately.

\subsection{Numerical experiments for example (a)}
We begin with the simple case $q=0$. By the representations \eqref{eqn:sols},
the exact solution $u(x)$ is given by
\begin{equation*}
  u(x) =  \frac{1}{\Gamma(\alpha+2)}(x^{\alpha-1}-x^{\alpha+1})-\frac{2}{\Gamma(\alpha+3)}(x^{\alpha-1}-x^{\alpha+2}),
\end{equation*}
and it belongs to $\Hdi0{\alpha-1+\beta}$ with $\beta\in[1-\alpha/2,1/2)$ due to the presence of the
term $x^{\al-1}$. Thus the standard Galerkin FEM converges slowly; see \cite[Table 1]{JinLazarovPasciak:2013a}
and Table \ref{tab:Linfsmooth-q=0}. In the table, \texttt{rate} refers to the empirical convergence rate when the
mesh size $h$ halves, and the numbers in the bracket denote theoretical rates. The results in Table
\ref{tab:Linfsmooth-q=0} indicate that the $L^\infty\II$ estimate in Theorem \ref{thm:linf} is sharp. By
the definition \eqref{eqn:reg+sing}, in the singularity splitting, the regular part $u^r$ and the singular
part $u^s$ are given respectively by
\begin{equation*}
  u^r=\frac{2(x^{\al+2}-x^2)}{\Gamma(\alpha+3)}+\frac{x^2-x^{\al+1}}{\Gamma(\alpha+2)}\quad\mbox{and}\quad
  \mu u^s=\left(\frac{1}{\Gamma(\alpha+2)}-\frac{2}{\Gamma(\alpha+3)}\right)(x^{\al-1}-x^2).
\end{equation*}
In particular, the regular part $u^r$ belongs to $H^2\II$. Since the singularity strength
$\mu=1/\Gamma(\alpha+2)-2/\Gamma(\alpha+3)$ does
not depend on the regular part $u^r$, it suffices to check errors for $e=u^r-u^r_h$.
In Table \ref{tab:smooth-q=0} we show the errors $\|e\|_{L^2\II}$, $\| e
\|_{\widetilde H^{\al/2}\II}$ and $\|e\|_{L^\infty\II}$. The numerical results show $O(h^{2})$, $O(h^{2-\al/2})$
and $O(h^2)$ for the $L^2\II$, $\Hd{\al/2}$ and $L^\infty\II$-norms of the error, respectively.
The $\Hd{\al/2}$ estimate is fully confirmed; however, the $L^2\II$ and $L^\infty\II$
estimates are suboptimal: the empirical ones are one half order higher than the theoretical ones.
The suboptimality is attributed to the low regularity of the adjoint problem
\eqref{eqn:dual}, used in Nitsche's trick.

\begin{table}[hbt!]
\caption{The $L^\infty\II$-norm of the error by the standard Galerkin approximation
for example (a) with $q=0$, $\al=1.25, 1.5, 1.75$, $h=1/2^k$.}\label{tab:Linfsmooth-q=0}
\vspace{-.3cm}
\begin{center}
     \begin{tabular}{|c|cccccc|c|}
     \hline
     $k$  &$5$ &$6$ & $7$ &$8$ &$9$ & $10$ &rate \\
     \hline
     $\al=1.25$     &2.91e-2 &2.44e-2 &2.05e-2 &1.73e-2 &1.45e-2 &1.22e-2 &$\approx$ 0.25 ($0.25$)\\
    \hline
     $\al=1.5$      &4.87e-3 &3.44e-3 &2.42e-3 &1.71e-3 &1.21e-3 &8.55e-4 &$\approx$ 0.50 ($0.50$)\\
     \hline
     $\al=1.75$     &7.46e-4 &4.37e-4 &2.59e-4 &1.54e-4 &9.16e-5 &5.44e-5 &$\approx$ 0.75 ($0.75$)\\
     \hline
     \end{tabular}
\end{center}
\end{table}

\begin{table}[hbt!]
\caption{The $L^2\II$-, $\tilde{H}^{\al/2}\II$- and $L^\infty\II$-norm of the error $e=u^r-u_h^r$ for
example (a) with $q=0$, $\al=1.25, 1.5, 1.75$, $h=1/2^k$.}
\label{tab:smooth-q=0}
\vspace{-.3cm}
\begin{center}
     \begin{tabular}{|c|c|cccccc|c|}
     \hline
     $\al$  & $k$ &$5$ &$6$ &$7$ & $8$ & $9$ & $10$  &rate \\
     \hline
     $1.25$ & $L^2$  &6.56e-5 &1.64e-5 &4.11e-6 &1.03e-6 &2.56e-7 &6.33e-8 &$\approx$ 2.00 ($1.50$) \\
     \cline{2-8}
     & $\tilde{H}^{\al/2}$  &2.98e-4 &1.11e-4 &4.23e-5 &1.62e-5 &6.21e-6 & 2.39e-6 &$\approx$ 1.36 ($1.38$)\\
     \cline{2-8}
     & $L^\infty$     &1.16e-4 &2.92e-5  &7.33e-6& 1.84e-6& 4.59e-7 &1.15e-7 & $\approx$ 2.00 ($1.50$)\\
    \hline
     $1.5$ & $L^2$    &3.62e-5 &9.16e-6 &2.31e-6 &5.79e-7 &1.45e-7 & 3.59e-8& $\approx$ 2.00 ($1.50$)\\
     \cline{2-8}
     & $\tilde{H}^{\al/2}$    &4.58e-4& 1.90e-4 &7.92e-5 &3.32e-5 &1.39e-5 &5.28e-6&$\approx$ 1.25 ($1.25$)\\
     \cline{2-8}
     & $L^\infty$      &7.58e-5 &1.92e-5 &4.81e-6 &1.21e-6 &3.02e-7 & 7.55e-8 &$\approx$ 1.99 ($1.50$)\\
     \hline
     $1.75$ & $L^2$     &1.59e-5 &4.11e-6 &1.05e-6 &2.69e-7 &6.84e-8 & 1.72e-8&$\approx$ 1.97 ($1.50$)\\
     \cline{2-8}
     & $\tilde{H}^{\al/2}$  &5.57e-4 &2.54e-4  &1.16e-4 &5.30e-5 &2.43e-5 &1.11e-5 &$\approx$ 1.13 ($1.13$)\\
     \cline{2-8}
     & $L^\infty$    &4.32e-5 &1.10e-5  &2.77e-6  &6.96e-7 &1.74e-7 & 4.36e-8 &$\approx$ 2.00 ($1.50$)\\
     \hline
     \end{tabular}
\end{center}
\end{table}

Next we check the problem with the potential $q=x(1-x)\in \Hd {1+\epsilon}$, $\epsilon\in[0,1/2)$.
Thus by Theorem \ref{thm:regrl-new}, the regular part $u^r$ belongs to $H^2\II$, and Theorems
\ref{lem:femerror} and \ref{thm:femerrorinf} predict almost $O(h^{3/2})$, $O(h^{2-\al/2})$
and $O(h^{3/2})$ for the $L^2\II$, $\Hd{\al/2}$,
and $L^\infty\II$-norms of the error, respectively. The numerical results fully confirm the $\Hd{\al/2}$ estimate,
but show one half order higher convergence for the $L^2\II$ and $L^\infty\II$-norm of the error, cf.
Table \ref{tab:smooth-q=smooth}. Further, we observe that the influence of the
potential term on the approximation error is negligible. The numerical results in Table
\ref{tab:smooth-q=smooth-la} show that the error $|\mu_h-\mu|$ of the
reconstructed singular strength $\mu$ achieves a second-order
convergence, which is higher than the theoretical rate from Lemma \ref{lem:singerror} by one half order.

\begin{table}[hbt!]
\caption{The $L^2\II$-, $\tilde{H}^{\al/2}\II$- and $L^\infty\II$-norm of the error $e=u^r-u^r_h$ for
example (a) with $q=x(1-x)$, $\al=1.25, 1.5, 1.75$, $h=1/2^k$.}\label{tab:smooth-q=smooth}
\vspace{-.3cm}
\begin{center}
     \begin{tabular}{|c|c|cccccc|c|}
     \hline
     $\al$  & $k$ & $5$ &$6$ &$7$ & $8$ & $9$ & $10$&rate \\
     \hline
     $1.25$ & $L^2$   &6.42e-5 &1.61e-5 &4.02e-6 &1.00e-6 &2.51e-7 &6.19e-8 &$\approx$ 2.01 ($1.50$) \\
     \cline{2-8}
     & $\tilde{H}^{\al/2}$   &2.80e-4 &1.04e-4 &3.96e-5 &1.51e-5 &5.81e-6 &2.23e-6&$\approx$ 1.38 ($1.38$)\\
     \cline{2-8}
     & $L^\infty$      &1.13e-4 &2.84e-5 &7.12e-6 &1.78e-6  &4.46e-7 &1.12e-7 &$\approx$ 2.00 ($1.50$)\\
    \hline
     $1.5$ & $L^2$   &3.50e-5 &8.86e-6 &2.23e-6 &5.61e-7 &1.40e-7 &3.48e-8 &$\approx$ 2.01 ($1.50$)\\
     \cline{2-8}
     & $\tilde{H}^{\al/2}$   &4.34e-4 &1.80e-4 &7.50e-5 &3.14e-5 &1.32e-5 &5.52e-6 &$\approx$ 1.25 ($1.25$)\\
     \cline{2-8}
     & $L^\infty$    &7.35e-5 &1.86e-5 &4.67e-6 &1.17e-6 &2.93e-7 &7.33e-8&$\approx$ 2.00 ($1.50$)\\
     \hline
     $1.75$ & $L^2$    &1.54e-5 &3.97e-6 &1.02e-6 &2.61e-7 &6.62e-8 &1.66e-8 &$\approx$ 1.97 ($1.50$)\\
     \cline{2-8}
     & $\tilde{H}^{\al/2}$   &5.34e-4 &2.43e-4 &1.11e-4 &5.08e-5 &2.33e-5 &1.06-5 &$\approx$ 1.13 ($1.13$)\\
     \cline{2-8}
     & $L^\infty$    &4.20e-5 &1.07e-5  &2.70e-6  &6.78e-7 &1.70e-7 &4.25e-8 &$\approx$ 2.00 ($1.50$)\\
     \hline
     \end{tabular}
\end{center}
\end{table}

\begin{table}[hbt!]
\caption{$|\mu-\mu_h|$ for example (a) with $q=x(1-x)$, $\al=1.25, 1.5, 1.75$, $h=1/2^k$.}\label{tab:smooth-q=smooth-la}
\vspace{-.3cm}
\begin{center}
     \begin{tabular}{|c|cccccc|c|}
     \hline
     $k$  &$5$ &$6$ & $7$ &$8$ &$9$ & $10$ &rate \\
     \hline
     $\al=1.25$    &8.62e-6 &2.16e-6 &5.40e-7 &1.35e-7 &3.33e-8 &7.93e-9  &$\approx$ 2.02 ($1.50$) \\
     \cline{1-8}
     $\al=1.5$     &3.70e-6 &9.43e-7 &2.39e-7 &6.01e-8 &1.49e-8 &3.57e-9  &$\approx$ 2.01 ($1.50$)\\
     \cline{1-8}
     $\al=1.75$    &9.49e-7 &2.60e-7  &6.96e-8 &1.83e-8 &4.72e-9 &1.16e-9  &$\approx$ 1.96 ($1.50$)\\
     \hline
     \end{tabular}
\end{center}
\end{table}

Apart from the reconstruction technique, there are alternative strategies for enhancing the computational
efficiency of the standard Galerkin FEM. Mesh grading is one such possible choice, and clearly,
it preserves the variational formulation. In Table
\ref{tab:smooth-q=smooth-meshgrad}, we present numerical results of the standard Galerkin FEM using graded
meshes. Specifically, we consider the grid $x_j=(jh)^\delta$ with some $\delta\ge1$, $j=0,1,\ldots,m$. The choice $\delta=1$
corresponds to a uniform mesh, whereas the choice $\delta > 1$ makes the mesh graded near $x = 0$, thereby
compensating the singularity of the solution. It is observed from Table \ref{tab:smooth-q=smooth-meshgrad}
that the optimal convergence rate is indeed achieved for a sufficiently large $\delta$. Further, its accuracy
is comparable with the reconstruction approach, cf. Table \ref{tab:smooth-q=smooth}. However, computationally,
for a graded mesh, it is far more expensive to form the stiffness matrix, since it lacks
nice structure, whereas for a uniform mesh, it is Toeplitz (plus rank-one perturbation). Further, in the mesh grading approach, the
singularity is hard wired in the implementation and cannot automatically adapted to the case of vanishing singularity:
if the solution is indeed smooth instead weakly singular, then the use of a graded mesh is wasteful; whereas
the reconstruction technique can adapt itself automatically to the case of a vanishing singularity by recovering $\mu=0$,
which makes it more flexible than the one based on graded meshes.

\begin{table}[hbt!]
\caption{The $L^2\II$-, $\tilde{H}^{\al/2}\II$- and $L^\infty\II$-norm of the error $u-u_h$ for
example (a) with $q=x(1-x)$, $\al=1.25, 1.75$, $h=1/2^k$ and $\delta=2,5$, using graded meshes.}\label{tab:smooth-q=smooth-meshgrad}
\vspace{-.3cm}
\begin{center}
     \begin{tabular}{|c|c|c|cccccc|c|}
     \hline
     $\al$  & $\delta$ &$k$ & $3$ &$4$ &$5$ & $6$ & $7$ & $8$&rate \\
     \hline
      & &$L^2$   &5.18e-3 &1.59e-3 &5.41e-4 &1.86e-4 &6.34e-5 &2.28e-5 &$\approx$ 1.56  \\
     \cline{3-9}
     &2 &$\tilde{H}^{\al/2}$   &5.67e-2 &4.59e-2 &3.70e-2 &2.89e-2 &2.05e-2 &1.30e-2&$\approx$ 0.42 \\
     \cline{3-9}
     $1.25$& &$L^\infty$      &2.46e-2 &1.64e-2 &1.15e-2 &8.10e-3  &5.73e-3 &4.05e-3 &$\approx$ 0.52 \\
    \cline{2-10}
      & &$L^2$   &6.63e-3 &3.04e-3 &9.08e-4 &2.40e-4 &6.13e-5 &1.57e-5 &$\approx$ 1.97 \\
     \cline{3-9}
     &5 &$\tilde{H}^{\al/2}$   &1.79e-2 &6.16e-3 &2.22e-3 &7.69e-4 &2.63e-4 &9.31e-5&$\approx$ 1.52 \\
     \cline{3-9}
     & &$L^\infty$      &9.65e-3 &5.00e-3 &1.61e-3 &4.42e-3  &1.15e-4 &2.98e-5 &$\approx$ 1.92 \\
    \hline
      && $L^2$    &3.47e-4 &8.24e-5 &2.06e-5 &5.29e-6 &1.38e-6 &3.58e-7 &$\approx$ 1.98 \\
     \cline{3-9}
     &2 &$\tilde{H}^{\al/2}$   &5.34e-4 &2.43e-4 &1.11e-4 &5.08e-5 &2.33e-5 &1.06-5 &$\approx$ 1.02 \\
     \cline{3-9}
     $1.75$& &$L^\infty$    &4.20e-5 &1.07e-5  &2.70e-6  &6.78e-7 &1.70e-7 &4.25e-8 &$\approx$ 1.70 \\
     \cline{2-10}
      & & $L^2$    &2.19e-3 &3.73e-4 &9.19e-5 &2.46e-5 &6.69e-6 &1.80e-6 &$\approx$ 2.05 \\
     \cline{3-9}
     & 5&$\tilde{H}^{\al/2}$ &1.41e-2 &6.53e-3 & 3.07e-3& 1.41e-3  &6.47e-4 &2.96e-4  &$\approx$ 1.11\\
     \cline{3-9}
     & &$L^\infty$  &3.90e-3 &9.29e-4 &1.88e-4  &7.25e-5 &2.19e-5  &5.93e-6   &$\approx$ 1.87\\
     \hline
     \end{tabular}
\end{center}
\end{table}

\subsection{Numerical experiments for example (b)}
In Tables \ref{tab:intermediate-q=smooth} and \ref{tab:intermediate-q=smoothla} we present numerical results for
problem (b) with $q=x(1-x)$. Since the source term $f$ is in $\Hd{\ep}$, $\ep \in [0,1/2)$, regular part $\ur$ is
in $H^2\II$ for $\al \in (3/2,2)$, whereas for $\al \in (1,3/2]$, $\ur$ is in the space $H^{\al+\ep}$ with $\ep \in
[0,1/2)$. In Table \ref{tab:intermediate-q=smooth}, the numerical results exhibit a convergence rate of the order
$O(h^{2})$ in the $L^2\II$-norm, and $O(h^{2-\al/2})$ in the $\Hd{\alpha/2}$-norm, for both $\al=1.5$ and $1.75$.
Further, we observe that the $L^2\II$ and $\Hd{\al/2}$-errors are respectively only of order $O(h^{1.75})$ and
$O(h^{1.13})$ for $\alpha=1.25$, due to a lack of regularity of the regular part $u^r$. The empirical $L^2\II$ rate
is one half order higher than the theoretical prediction. Further, for the case of $\alpha=1.75$, we observe a superconvergence
phenomenon for the $L^\infty\II$-estimate, whereas for the case $\alpha=1.5$ and $1.25$, the $L^\infty\II$-estimate
agrees with our convergence theory. These experiments indicate an $L^\infty\II$-norm of the error in the form $O(h^{\min(2,\alpha+\ell(\beta,
\gamma)-1/2)})$, instead of $O(h^{\min(2,\al+\ell(\beta,\gamma))-1+\beta})$ from Theorem \ref{thm:femerrorinf}, and
a superconvergence phenomenon occurs if $\alpha+\ell(\beta,\gamma)>2$. Surprisingly, the error $|\mu-\mu_h|$ of the
recovered singularity strength $\mu_h$ achieves a
second-order convergence even for $\al=1.25$, cf. Table \ref{tab:intermediate-q=smoothla}, which remains to be justified.

\begin{table}[hbt!]
\caption{The $L^2\II$-, $\tilde{H}^{\al/2}\II$- and $L^\infty\II$-norm of the error $e=u^r-u_h^r$ for example
(b) with $q=x(1-x)$, $\al=1.25, 1.5, 1.75$, $h=1/2^k$.} \label{tab:intermediate-q=smooth}
\vspace{-.3cm}
\begin{center}
     \begin{tabular}{|c|c|cccccc|c|}
     \hline
     $\al$  & $k$ & $5$ &$6$ &$7$ & $8$ & $9$ & $10$ &rate \\
     \hline
     $1.25$ & $L^2$          &1.45e-4 &4.17e-5 &1.21e-5 &3.52e-6 &1.03e-6 &3.03e-7 &$\approx$ 1.77 ($1.25$) \\
     \cline{2-8}
     & $\tilde{H}^{\al/2}$   &1.51e-3 &7.11e-4 &3.33e-4 &1.55e-4 &7.18e-5 &3.33e-5 &$\approx$ 1.11 ($1.13$)\\
     \cline{2-8}
     & $L^\infty$            &8.59e-4 &3.89e-4 &1.71e-4 &7.36e-5  &3.14e-5 &1.33e-5&$\approx$ 1.23 ($1.25$)\\
    \hline
     $1.5$ & $L^2$           &7.18e-5 &1.87e-5 &4.83e-6 &1.25e-6 &3.20e-7 &8.16e-8 &$\approx$ 1.96 ($1.50$)\\
     \cline{2-8}
     & $\tilde{H}^{\al/2}$   &1.93e-3 &8.85e-4 &4.02e-4 &1.81e-4 &8.06e-5 &3.57e-5 &$\approx$ 1.17 ($1.25$)\\
     \cline{2-8}
     & $L^\infty$           &3.55e-4 &1.38e-4 &5.21e-5 &1.92e-5 &7.01e-6 &2.52e-6 &$\approx$ 1.47 ($1.50$)\\
     \hline
     $1.75$ & $L^2$          &2.78e-5 &6.93e-6 &1.72e-6 &4.24e-7 &1.04e-7 &2.55e-8 &$\approx$ 2.02 ($1.50$)\\
     \cline{2-8}
     & $\tilde{H}^{\al/2}$   &1.57e-3 &7.41e-4 &3.47e-4 &1.62e-4 &7.48e-5 &3.44e-5 &$\approx$ 1.11 ($1.13$)\\
     \cline{2-8}
     & $L^\infty$           &1.19e-4 &4.11e-5 &1.36e-5 &4.39e-6 &1.39e-6 &4.36e-7 &$\approx$ 1.68 ($1.50$)\\
     \hline
     \end{tabular}
\end{center}
\end{table}

\begin{table}[hbt!]
\caption{$|\mu-\mu_h|$ for example (b) with $q=x(1-x)$, $\al=1.25, 1.5, 1.75$, $h=1/2^k$.} \label{tab:intermediate-q=smoothla}
\vspace{-.3cm}
\begin{center}
     \begin{tabular}{|c|cccccc|c|}
     \hline
     $k$ & $5$ &$6$ & $7$ &$8$ &$9$ & $10$ &rate \\
     \hline
     $\al=1.25$    &1.25e-5 &3.11e-6 &7.73e-7 &1.92e-7 &4.75e-8 &1.13e-8 &$\approx$ 2.02 ($1.25$) \\
     \cline{1-8}
     $\al=1.5$    &5.76e-6 &1.41e-6 &3.48e-7 &8.60e-8 &2.11e-8 &5.00e-9 &$\approx$ 2.04 ($1.50$)\\
     \cline{1-8}
     $\al=1.75$   &1.85e-6 &4.52e-7 &1.11e-7 &2.72e-8 &6.63e-9 &1.56e-9 &$\approx$ 2.04 ($1.50$)\\
     \hline
     \end{tabular}
\end{center}
\end{table}

\subsection{Numerical experiments for example (c)}
In Tables \ref{tab:nonsmooth-q=smooth} and \ref{tab:nonsmooth-q=smoothla}, we present numerical results
for example (c) with $q=x(1-x)$. Since the source term $f(x)=x^{-1/4}$ is in $\Hd{\ep}$ with $\ep \in [0,1/4)$,
by Theorem \ref{thm:regrl-new}, the regular part $\ur$ is in $H^{\al +\ep}\II$, which implies a
convergence rate $O(h^{\al+\ep-1/2})$, $O(h^{\al/2+\ep})$ and $O(h^{\al+\ep-1/2})$ of the $L^2\II$-,
$\Hd{\al/2}$- and $L^\infty\II$-norm of the error, respectively, cf. Theorems \ref{lem:femerror} and
\ref{thm:femerrorinf}. The $L^2\II$, $\Hd{\al/2}$ and $L^\infty\II$-errors achieve
a rate $O(h^{\al+1/4})$, $O(h^{\al/2+1/4})$ and $O(h^{\al-1/4})$, respectively, for $\alpha=1.50$
and $1.75$. The convergence of the Galerkin approach slows down as the fractional order $\alpha$ tends to unity, due to the lower regularity pickup
of the regular part $u^r$, but the empirical behavior still agrees well with the theoretical predictions. However, for $\al=1.25$, the
approximation $u_h$ converges faster: a second-order convergence in both $L^2\II$ and
$L^\infty\II$-norms, and an $O(h^{1.38})$ rate in the $\Hd{\al/2}$-norm. This is attributed to the fact that
${_0I_x^{5/4}} x^{-1/4} = \Gamma(3/4)x$, which is actually much smoother than
$\Hdi 0{5/4+\ep}$, with $\ep\in[0,1/4)$, from Theorem \ref{thm:regrl-new}.
Interestingly, the approximation $\mu_h$ of the singularity strength $\mu$ is second order
accurate, cf. Table \ref{tab:nonsmooth-q=smoothla}, despite the low regularity of the data.

\begin{table}[hbt!]
\caption{The $L^2\II$-, $\tilde{H}^{\al/2}\II$- and $L^\infty\II$-norms of the error $e=u^r-u^r_h$ for
example (c) with $q=x(1-x)$, $\al=1.25, 1.5, 1.75$, $h=1/2^k$.} \label{tab:nonsmooth-q=smooth}
\vspace{-.3cm}
\begin{center}
     \begin{tabular}{|c|c|cccccc|c|}
     \hline
     $\al$  & $k$ & $5$ &$6$ &$7$ & $8$ & $9$ & $10$&rate \\
     \hline
     $1.25$ & $L^2$      &1.76e-4 &4.39e-5 &1.10e-5 &2.74e-6 &6.82e-7 &1.69e-7 &$\approx$ 2.00 ($1.00$) \\
     \cline{2-8}
     & $\tilde{H}^{\al/2}$     &1.61e-3 &6.16e-4 &2.36e-4 &9.09e-5 &3.50e-5 &1.35e-5 &$\approx$ 1.38 ($0.88$)\\
     \cline{2-8}
     & $L^\infty$              &2.76e-4 &6.92e-5 &1.73e-5 &4.33e-6 &1.08e-6 &2.71e-7 &$\approx$ 2.00 ($1.00$)\\
    \hline
     $1.5$ & $L^2$             &1.53e-4 &4.32e-5 &1.23e-5 &3.58e-6 &1.04e-6 &3.05e-7 &$\approx$ 1.76 ($1.25$)\\
     \cline{2-8}
     & $\tilde{H}^{\al/2}$     &4.18e-3 &2.11e-3 &1.07e-3 &5.40e-4 &2.73e-4 &1.38e-4 &$\approx$ 0.98 ($1.00$)\\
     \cline{2-8}
     & $L^\infty$              & 7.25e-4& 3.33e-4 &1.48e-4 &6.41e-5 &2.74e-5 &1.16e-5  &$\approx$ 1.21 ($1.25$)\\
     \hline
     $1.75$ & $L^2$               &8.91e-5 &2.33e-5  &6.05e-6 &1.56e-6 &4.02e-7 &1.02e-7 &$\approx$ 1.95 ($1.50$)\\
     \cline{2-8}
     & $\tilde{H}^{\al/2}$      &5.16e-3  &2.59e-3 &1.28e-3 &6.31e-4 &3.08e-4 &1.49e-4 &$\approx$ 1.05 ($1.13$)\\
     \cline{2-8}
     & $L^\infty$                &4.47e-4 &1.76e-4  &6.67e-5  &2.47e-5  &9.03e-6  &3.29e-7  &$\approx$ 1.45 ($1.50$)\\
     \hline
     \end{tabular}
\end{center}
\end{table}

\begin{table}[hbt!]
 \caption{$|\mu-\mu_h|$ for example (c) with $q=x(1-x)$, $\al=1.25, 1.5, 1.75$, $h=1/2^k$.} \label{tab:nonsmooth-q=smoothla}
\vspace{-.3cm}
\begin{center}
     \begin{tabular}{|c|cccccc|c|}
     \hline
     $k$ & $5$ &$6$ & $7$ &$8$ &$9$ & $10$ &rate \\
     \hline
     $\al=1.25$  &2.10e-5 &5.25e-6 &1.31e-6 &3.27e-7 &8.08e-8 &1.92e-8 &$\approx$ 2.03 ($1.00$) \\
     \cline{1-8}
     $\al=1.5$          &1.18e-5 &2.90e-6 &7.13e-7 &1.76e-7 &4.31e-8 &1.02e-8 &$\approx$ 2.05 ($1.50$)\\
     \cline{1-8}
     $\al=1.75$        &5.59e-6 &1.36e-6  &3.33e-7 &8.13e-8 &2.00e-8 &4.87e-9  &$\approx$ 2.03 ($1.50$)\\
     \hline
     \end{tabular}
\end{center}
\end{table}

\subsection{Numerical results for mixed boundary condition}
Last we illustrate the reconstruction technique for the mixed boundary condition, cf. Section \ref{sec:neumann}.
In Tables \ref{tab:Neumann-u} and \ref{tab:Neumann-mu}, we present numerical results for example (c) with $q=x
(1-x)$ and $\DDR 0 {\alpha-1} u(0)=0$ and $u(1)=0$, for three different $\alpha$ values. We observe a convergence rate of the order $O(h^{\min(2,
\al+\ell_n(\beta,\gamma))})$ in the $L^2\II$-norm, $O(h^{\min(2,\ell_n(\beta,\gamma))-\al/2})$ in the $\Hd{\alpha/2}$-norm
and $O(h^{\min(2,\al+\ell_n(\beta,\gamma))-1/2})$ in the $L^\infty\II$-norm for all $\al\in(3/2,2)$, which confirms
Theorem \ref{thm:regerror-Neumann}. Like before,
the approximate singularity strength $\mu_h$ achieves a second-order convergence even for $\al$ close to 1.5, cf. Table
\ref{tab:Neumann-mu}, which remains to be theoretically justified.

\begin{table}[hbt!]
\caption{The $L^2\II$-, $\tilde{H}^{\al/2}\II$- and $L^\infty\II$-norms of the error $ e= u^r-u^r_h$ for
example (c) with mixed boundary conditions, $q=x(1-x)$, $\al=1.6, 1.75, 1.9$, $h=1/2^k$.} \label{tab:Neumann-u}
\vspace{-.3cm}
\begin{center}
     \begin{tabular}{|c|c|cccccc|c|}
     \hline
     $\al$  & $k$ & $5$ &$6$ &$7$ & $8$ & $9$ & $10$&rate \\
     \hline
     $1.6$ & $L^2$      &1.14e-4 &3.19e-5 &8.89e-6 &2.47e-6 &6.86e-7 &1.90e-7 &$\approx$ 1.85 ($1.35$) \\
     \cline{2-8}
     & $\tilde{H}^{\al/2}$     &4.51e-3 &2.31e-3 &1.17e-3 &5.85e-4 &2.91e-4 &1.44e-4 &$\approx$ 0.98 ($0.90$)\\
     \cline{2-8}
     & $L^\infty$              &6.06e-4 &2.64e-4 &1.10e-4 &4.48e-5 &1.78e-5 &7.19e-6 &$\approx$ 1.31 ($1.35$)\\
    \hline
     $1.75$ & $L^2$             &8.24e-5 &2.17e-5 &5.68e-6 &1.47e-6 &3.80e-7 &9.73e-8 &$\approx$ 1.95 ($1.50$)\\
     \cline{2-8}
     & $\tilde{H}^{\al/2}$     &4.95e-3 &2.50e-3 &1.25e-3 &6.17e-4 &3.02e-4 &1.46e-4 &$\approx$ 1.03 ($1.13$)\\
     \cline{2-8}
     & $L^\infty$              & 4.51e-4& 1.78e-4 &6.72e-5 &2.49e-5 &9.07e-6 &3.30e-6  &$\approx$ 1.45 ($1.50$)\\
     \hline
      $1.9$ & $L^2$       &5.50e-5 &1.42e-5  &3.61e-6 &9.10e-7 &2.28e-7 &5.63e-8 &$\approx$ 1.99 ($1.50$)\\
     \cline{2-8}
     & $\tilde{H}^{\al/2}$      &4.68e-3  &2.39e-3 &1.20e-3 &6.00e-4 &2.97e-4 &1.46e-4 &$\approx$ 1.00 ($1.05$)\\
     \cline{2-8}
     & $L^\infty$                &2.82e-4 &1.02e-4  &3.54e-5  &1.20e-5  &4.02e-6  &1.31e-6  &$\approx$ 1.56 ($1.50$)\\
     \hline
     \end{tabular}
\end{center}
\end{table}

\begin{table}[hbt!]
 \caption{$|\mu-\mu_h|$ for example (c) with mixed boundary conditions, $q=x(1-x)$, $\al=1.6, 1.75, 1.9$, $h=1/2^k$.} \label{tab:Neumann-mu}
\vspace{-.3cm}
\begin{center}
     \begin{tabular}{|c|cccccc|c|}
     \hline
     $\alpha\backslash k$ & $5$ &$6$ & $7$ &$8$ &$9$ & $10$ &rate \\
     \hline
     $1.6$  &3.69e-5 &3.40e-6 &4.00e-7 &9.82e-8 &2.41e-8 &5.88e-9 &$\approx$ 2.03 ($1.00$) \\
     \cline{1-8}
     $1.75$          &2.07e-5 &1.11e-6 &2.68e-7 &6.53e-8 &1.59e-8 &3.83e-8 &$\approx$ 2.04 ($1.50$)\\
     \cline{1-8}
     $1.9$        &5.59e-6 &1.36e-6  &3.33e-7 &8.13e-8 &2.00e-8 &4.87e-9  &$\approx$ 2.02 ($1.50$)\\
     \hline
     \end{tabular}
\end{center}
\end{table}

\section{conclusion}\label{sec:conclusion}
In this work, we have developed and analyzed a new finite element technique for approximating boundary value problems
with a Riemann-Liouville fractional derivative in the leading term. It relies on splitting the solution into a regular
part and a singular part, where the regular part lies in $H^{\al+\ell(\beta,\gamma)}\II\cap \Hd{\al/2}$. We have derived
a new variational formulation for the regular part, and established its well-posedness and enhanced regularity. Further,
a Galerkin finite element approximation for the regular part and a reconstruction formula for the singular part have
been proposed. The stability of the discrete variational formulation, and error estimates of the regular part in
$\Hd{\alpha/2}$, $L^2\II$, and $L^\infty\II$-norms, and the reconstructed singularity strength were established, which
are higher than that for the standard Galerkin FEM approximation. The idea can be extended straightforwardly
to other type of boundary conditions, and has also been illustrated on the mixed boundary condition.

Numerical experiments with smooth and nonsmooth source terms fully confirmed the convergence of the
numerical scheme. Numerically, the $\Hd{\alpha/2}$-estimate agree excellently with the theoretical
ones. However, the $L^2\II$-estimate is suboptimal: the empirical rate is one-half order higher than the theoretical one. This
suboptimality is attributed to the low regularity of the adjoint solution. This has been observed
earlier in the context of Poisson's equation on L-shaped domains \cite{CaiKim:2001} and the standard
Galerkin method for fractional boundary value problems \cite{JinLazarovPasciak:2013a}. Further, the
$L^\infty\II$-error exhibits a superconvergence phenomenon in case of $\alpha+\ell(\beta,\gamma)>2$.
The reconstructed strength $\mu_h$ always achieves a second-order accuracy, irrespective of the
fractional order $\alpha$ and the smoothness of the source term, which is better than the estimate
in Theorem \ref{lem:singerror}. Optimal convergence rates in the $L^2\II$ and $L^\infty\II$-norms and
that for the singularity strength still await mathematical justifications, for both Dirichlet
and mixed problems. Last, it is of immense interest to extend
the singularity approach to the multi-dimensional case and the mixed case involving both left- and right-sided
fractional derivatives, for which however the solution theory remains
to be developed.

\section*{Acknowledgements}
The authors are grateful to two anonymous referees for their helpful comments, which have led to
an improved presentation of the paper. The research of B. Jin has been partly supported by NSF
Grant DMS-1319052 and National Science Foundation of China grant (No.11471141).

\bibliographystyle{abbrv}
\bibliography{frac}
\end{document}